\documentclass[a4paper,fleqn]{cas-dc}

\usepackage[authoryear]{natbib}

\usepackage{bm}
\usepackage{amsmath}
\usepackage{tabularray}
\usepackage{array}
\usepackage{pifont}
\usepackage{bbm}
\usepackage{booktabs}
\usepackage{algorithmic}
\usepackage{algorithm}
\usepackage{CJK}
\usepackage{enumitem}
\usepackage{amsthm}
\newtheorem{assumption}{Assumption}
\newtheorem{theorem}{Theorem}
\newtheorem{lemma}{Lemma}
\newtheorem*{remark}{Remark}
\newtheorem*{notation}{Notation}
\usepackage{multirow}
\usepackage{caption}
\usepackage{subfig}
\usepackage{hyperref}
\hypersetup{
	colorlinks=true,
	linkcolor=cyan,
	citecolor=cyan,
	urlcolor=cyan,
}

\def\tsc#1{\csdef{#1}{\textsc{\lowercase{#1}}\xspace}}
\tsc{WGM}
\tsc{QE}
\tsc{EP}
\tsc{PMS}
\tsc{BEC}
\tsc{DE}
\begin{document}
\let\WriteBookmarks\relax
\def\floatpagepagefraction{1}
\def\textpagefraction{.001}
\shorttitle{Visual Localization Algorithm}
\shortauthors{Jindi Zhong et~al.}

\title [mode = title]{An Efficient Algorithm for Learning-Based Visual Localization}

% 第一作者
\author[1]{Jindi Zhong} \ead{zjd@sdust.edu.cn}
\credit{Conceptualization, Methodology, Software Development, Investigation, Formal Analysis, Writing – Original Draft}
% 第二作者  
\author[1]{Ziyuan Guo} \ead{skdgzy@sdust.edu.cn}
\credit{Algorithm Implementation, Writing – Review and Editing}

\author[1]{Hongxia Wang} \ead{whx1123@126.com}
\credit{Data Validation, Writing – Review and Editing}

%通讯作者
\author[1, 2]{Huanshui Zhang}[orcid=0000-0002-8611-7327] \cormark[1] \ead{hszhang@sdu.edu.cn}
\credit{Supervision, Funding Acquisition, Project Administration, Correspondence}

% 单位1
\affiliation[1]{
	organization={College of Electrical Engineering and Automation, Shandong University of Science and Technology},
	addressline={No. 579, Qianwangang Road},
	city={Qingdao},
	postcode={266590},
	country={China}
}

% 单位2  
\affiliation[2]{
	organization={School of Control Science and Engineering, Shandong University},
	addressline={No. 17923, Jingshi Road},
	city={Jinan}, 
	postcode={250061},
	country={China}
}

% 通讯作者标记
\cortext[1]{Corresponding author}

\begin{abstract}
This paper addresses the visual localization problem in Global Positioning System (GPS)-denied environments, where computational resources are often limited. To achieve efficient and robust performance under these constraints, we propose a novel algorithm. The algorithm stems from the optimal control principle (OCP). It incorporates diagonal information estimation of the Hessian matrix, which results in training a higher-performance deep neural network and accelerates optimization convergence. Experimental results on public datasets demonstrate that the final model achieves competitive localization accuracy and exhibits remarkable generalization capability. This study provides new insights for developing high-performance offline positioning systems.
\end{abstract}


\maketitle

\section{Introduction}
\label{Introduction}
Visual localization, the task of estimating geographic coordinates or Six Degrees of Freedom poses from images, is a cornerstone of autonomous systems such as self-driving cars \citep{ref1} and augmented reality applications \citep{ref2}. While GPS and Light Detection and Ranging provide coarse localization, they suffer from signal loss (e.g., in urban canyons) and incur high costs. In contrast, vision-based methods offer a low-cost and universally applicable alternative by leveraging ubiquitous cameras.

While traditional geometry-based visual localization methods achieve high accuracy, they often lack robustness in dynamic environments.
Meanwhile, existing learning-based visual localization approaches also face a fundamental challenge: complex architectures achieve high precision but are computationally expensive, while simpler architectures are efficient but often lack the robustness for accurate prediction \citep{ref10}. This limitation motivates our work on developing efficient optimization strategies specifically tailored for resource-limited deployment.

Therefore, to holistically address this challenge, we propose an efficient learning method that integrates a novel optimization algorithm with a clean Convolutional Neural Network (CNN) architecture. This architecture, containing less than 1$\%$ of the parameters of ResNet-18 \citep{ref11}, forms a crucial part of our framework. Our approach is motivated by three key factors:
\begin{itemize}
	\item[]
	$\bullet$ To ensure a fair and focused evaluation of the optimization algorithm, isolated from the performance gains of large-scale models.
	
	$\bullet$ To directly address the practical constraints of data availability and hardware limitations, which are critical for real-world edge deployment \citep{ref10, ref12}.
	
	$\bullet$ To build on the observation that robust visual localization primarily depends on coarse geometric priors \citep{ref13} rather than excessively complex feature representations, thereby making a simple network adequate when combined with an effective optimizer.
\end{itemize}

To this end, we develop a novel second-order optimization algorithm grounded in the OCP. The core of our approach involves an efficient diagonal approximation of the Hessian matrix using Hutchinson's method, enhanced with exponential moving averages for both gradients and Hessian diagonals, and an adaptive step-size strategy. This design aims to navigate the non-convex loss landscape effectively while maintaining computational feasibility for resource-constrained devices.

The main contributions of this paper are summarized as follows:
\begin{itemize}
	\item[]
	$\bullet$ We propose a novel optimization algorithm based on the OCP, which incorporates an efficient diagonal approximation of the Hessian matrix.
	
	$\bullet$ We provide a theoretical analysis of the algorithm's convergence rate. Under non-convex settings, it is shown to achieve a convergence form similar to Adam-type algorithms, yet with a faster convergence rate.
	
	$\bullet$ Experiments on public datasets demonstrate that the proposed algorithm achieves competitive localization accuracy and remarkable generalization capability.
\end{itemize}

The remainder of this paper is organized as follows. \autoref{Related Work} reviews related work on visual localization and optimization algorithms. \autoref{Problem Formulation} formulates the problem. \autoref{The Optimization Algorithm} provides a detailed exposition of the OCP based algorithm design. \autoref{Simulation} presents experimental results and analysis. Finally, \autoref{Conclusion} concludes the paper with directions for future work.

\section{Related Work}
\label{Related Work}
\subsection{Visual Localization Methods}
Current visual localization systems can be broadly categorized into geometry-based and learning-based approaches. Geometry-based methods, such as Structure-from-Motion \citep{ref3} and Simultaneous Localization and Mapping \citep{ref4}, achieve high-precision localization by leveraging accurate feature correspondences and have demonstrated strong performance in many scenarios. However, they often face significant challenges in dynamic environments, where moving objects interfere with feature matching, and under appearance variations caused by changes in illumination, weather, or viewpoint.

These limitations have motivated the exploration of learning-based alternatives, which aim to enhance robustness by directly modeling visual cues rather than relying on handcrafted geometric pipelines. Within learning-based approaches, end-to-end deep learning approaches are particularly promising, as they bypass the fragility of traditional geometric pipelines (e.g., feature matching failures in texture-less regions) by directly learning visual patterns correlated with spatial coordinates \citep{ref5}. However, these deep learning solutions face their own significant challenges. Although complex architectures employing attention mechanisms \citep{ref7} and multi-scale fusion \citep{ref8} achieve high precision, they typically require massive annotated datasets, involve over 10 million parameters, and are highly sensitive to optimization strategies due to the non-convex loss landscape of coordinate regression \citep{ref9}.

The substantial computational and data requirements of complex architectures have spurred growing interest in the development of efficient and lightweight architectures for visual localization. This trend is primarily driven by the need for deployment on edge devices, which are subject to stringent constraints on memory, power consumption, and computational capabilities \citep{ref28}. Research in this domain generally follows two main approaches: architectural design and model compression. While these methods effectively address issues related to parameter volume and computational efficiency, they introduce a fundamental trade-off: a significant reduction in model capacity often leads to deteriorated localization accuracy. Consequently, the core challenge has shifted from computational efficiency alone to achieving an optimal balance between efficiency and accuracy under stringent resource constraints.

The performance of lightweight models is inherently constrained by the optimization process itself, making mere model reduction insufficient for achieving an optimal balance between accuracy and efficiency. Due to their limited capacity, lightweight architectures are inherently unable to capture complex features, necessitating highly effective optimizers to navigate their loss landscapes and maximize the potential of their limited parameters. This indicates that optimization is no longer just an implementation detail in the training process, but rather a critical enabler for practical lightweight visual localization. This insight directly motivates the present research.

\subsection{Optimization Algorithms}

The training effectiveness of deep learning models in visual localization is fundamentally dependent on the optimization algorithms employed.

Training CNNs to directly regress geographic coordinates from a single image constitutes an inherently complex, high-dimensional, and non-convex regression problem \citep{ref14}. The severe non-linearity arises from entangled factors including viewpoint variations, illumination changes, and scene geometry distortions, creating a rugged loss landscape with numerous suboptimal basins.

While first-order optimizers dominate due to their computational efficiency, their reliance solely on gradient information fundamentally limits their efficacy for high-precision coordinate regression. For example:
\begin{itemize}
	\item []
	$\bullet$ Stochastic Gradient Descent (SGD) \citep{ref15}: Suffers from slow convergence and sensitivity to learning rate tuning, often stagnating in shallow minima.
	
	$\bullet$ Adaptive Moment Estimation (Adam) \citep{ref16}: Mitigates some issues using adaptive learning rates, but exhibits biased gradient estimates and tends to converge to sharp minima \citep{ref17}, compromising pose accuracy.
	
	$\bullet$ Rectified Adam (RAdam) \citep{ref18}: Rectifies Adam's convergence instability early in training, but still fails to capture curvature information, limiting the final regression precision.
\end{itemize}

Second-order optimization addresses these limitations by incorporating the Hessian matrix, which encodes the local curvature of the loss function, delivering quadratic convergence and precise navigation of ravines \citep{ref19}. However, exact Hessian computation is infeasible for deep networks due to:
\begin{itemize}
	\item []
	$\bullet$ For high-dimensional optimization with $d$ parameters, the memory complexity scales as $\mathcal{O}(d^2)$.
	
	$\bullet$ Each iteration incurs an inversion cost of $\mathcal{O}(d^3)$.
	
	$\bullet$ In non-convex landscapes, the Hessian may be non-positive definite, potentially hindering convergence.
\end{itemize}
Practical approximations bridge this gap. For example:
\begin{itemize}
	\item []
	$\bullet$ Quasi-Newton methods \citep{ref20}: Approximate the inverse of the Hessian matrix iteratively but struggle with stochastic gradients.
	
	$\bullet$ Kronecker-factored approximations \citep{ref21}: Capture block-diagonal curvature for layers but scale poorly to large architectures.
	
	$\bullet$ Shampoo \citep{ref22}: Uses full-matrix preconditioning per tensor dimension but incurs high computational overhead.
\end{itemize}

In pursuit of a more efficient trade-off, diagonal Hessian approximations have emerged as a promising solution. The key advantage is that this method reduces memory complexity to $\mathcal{O}(d)$ and inversion cost to $\mathcal{O}(d)$. Hutchinson's method \citep{ref23} achieves an unbiased stochastic estimation of the diagonal or trace by combining Rademacher random vectors with matrix-vector multiplication, followed by element-wise multiplication and expectation calculation. This approach avoids the computational cost of full matrix operations on $\mathcal{O}(d^2)$, requiring only the storage and computational cost associated with $\mathcal{O}(d)$. 
This underpins modern methods such as AdaHessian \citep{ref24}, which achieve scalability through layer-wise diagonal approximation and adaptive learning rates. Although AdaHessian provides a scalable second-order optimization solution by employing a Hutchinson estimator-based diagonal Hessian approximation, its approach still has key limitations. Analysis reveals that AdaHessian's fixed approximation granularity (e.g., spatial averaging) and global learning rate adjustment mechanism remain suboptimal for visual coordinate regression tasks.

The algorithm proposed in this paper attempts to achieve improved solutions by preserving the distinctive properties of the Hessian diagonal elements and adopting an adaptive step-size strategy. This algorithm is developed based on the maximum principle, resulting in an iteration sequence that closely follows the optimal state trajectory of the corresponding OCP. The OCP method \citep{ref25} exhibits superlinear convergence under non-convex conditions. However, because its update process involves the inversion of the Hessian matrix, \cite{ref26} reduced the computational burden by replacing the matrix inversion with a tunable parameter matrix, and the modified method still demonstrates superlinear convergence in non-convex optimization. Nevertheless, under deep learning frameworks, the high computational and storage costs associated with the Hessian matrix make it difficult to apply such methods to large-scale parameter problems. Building upon the work of \cite{ref26}, this paper proposes an algorithm tailored for large-scale non-convex optimization problems. Specifically, the algorithm employs the Hutchinson's method to approximate the diagonal elements of the Hessian matrix, applies exponential moving averages to both the gradients and the approximated Hessian diagonals, and incorporates an adaptive step-size strategy to enhance robustness. The algorithm draws on the convergence rate of the generalized Adam framework \citep{ref27}, and under reasonable assumptions, achieves sublinear convergence.

\begin{notation}
	$\mathcal{O}$ describes the asymptotic order of a quantity, such as the computational/memory complexity of an algorithm (e.g., $\mathcal{O}(d)$) or the convergence rate of an optimization algorithm (e.g., $\mathcal{O}(1/T)$). $\|A\|$ denotes the Euclidean norm of a vector $A$. $\mathbb{E}[A]$ denotes the expectation of $A$. $|A|$ denotes the absolute value of $A$. $\langle \cdot, \cdot \rangle$ denotes the inner product of two vectors. $[\cdot]_{ij}$ represents the element in the $i$-th row and $j$-th column of a matrix. $\mathbb{R}$ denotes the set of real numbers. $I$ denotes the identity matrix. $\nabla f(x)$ and $\nabla^{2} f(x)$ represent the exact gradient and the exact Hessian matrix, respectively.
\end{notation}

\section{Problem Formulation}
\label{Problem Formulation}
This section establishes the problem formulation for the subsequently proposed OCP based algorithm. It begins by introducing the structure of the neural network, then focuses on elucidating how to transform it into a solvable optimization problem.

\subsection{CNN Architecture Overview}
The CNN consists of four convolutional layers with increasing depth. Each convolutional layer is followed by a rectified linear unit (ReLU) activation and a batch normalization layer to enhance feature representation and accelerate model convergence. Max pooling layers are incorporated to reduce the spatial dimensions while preserving key semantic information. To prevent overfitting, dropout regularization is applied with a certain probability. Finally, the convolutional feature maps are flattened and passed through a fully connected layer, which serves as the regression head to output continuous geographic coordinates.

The CNN architecture is illustrated in the figure below.
\begin{comment}
	\begin{figure}[H]
		\centering
		\includegraphics[width=3in]{fig1}
		\caption{Simulation results for the network.}
		\label{fig_1}
	\end{figure}
\end{comment}
\begin{figure}[h]
	\centering
	\includegraphics[width=3in]{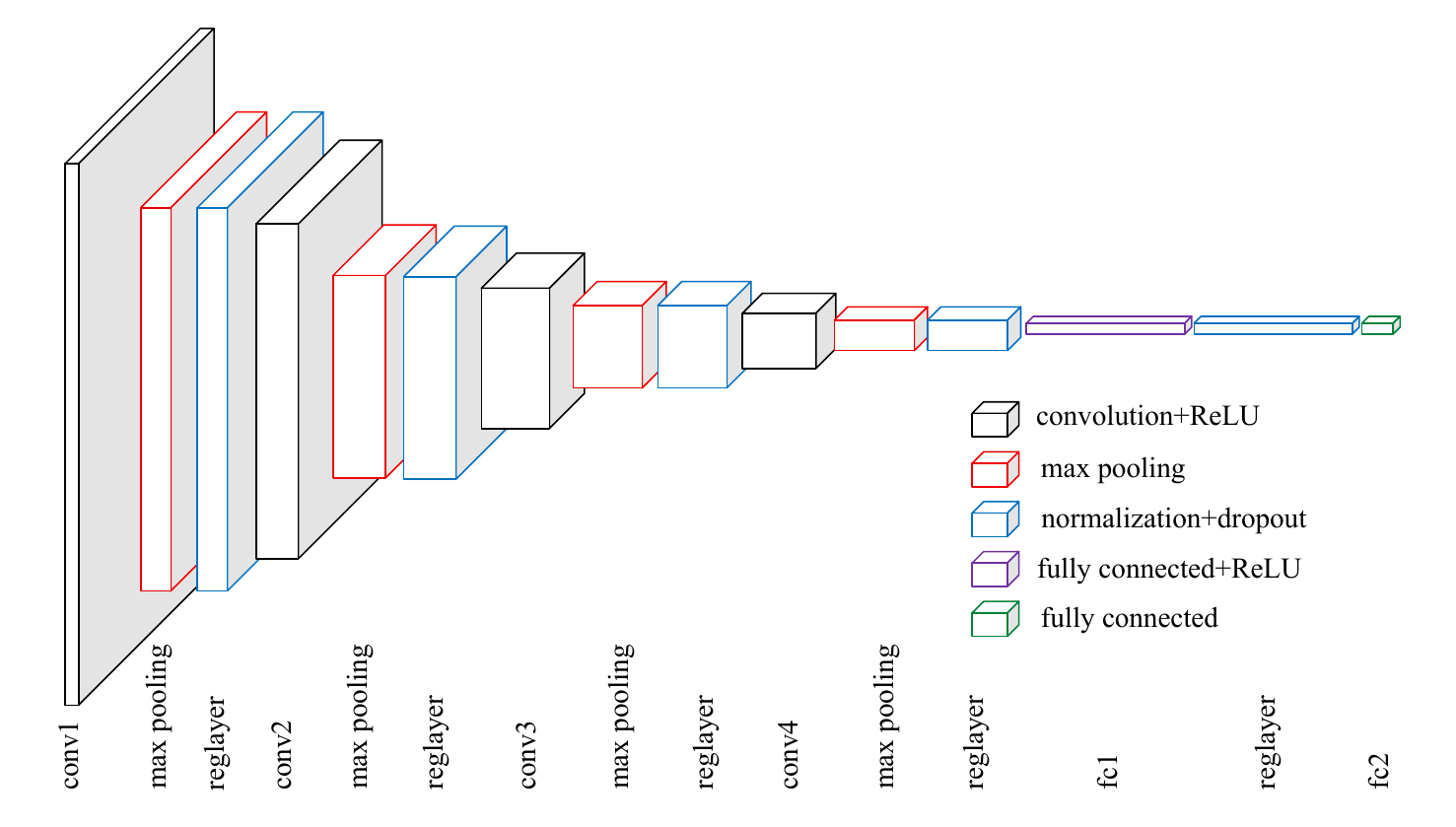}
	\caption{The CNN architecture}
	\label{cnn}
\end{figure}
The detailed information of each layer, including the number of input and output channels, the kernel size, and the filter stride, is summarized in the Table \ref{cnn:architecture} below.
\begin{table}[ht]
	\setlength{\tabcolsep}{2pt}
	\centering
	\caption{Information of Each Layer in The CNN Model}
	\label{cnn:architecture}
	\begin{tabular}{ c|c| c| c| c| c| c }
		\toprule
		Name & \multicolumn{1}{c|}{Input} & \multicolumn{1}{c|}{Output} &
		\multicolumn{1}{c|}{Kernel size} & \multicolumn{1}{c|}{Stride} & \multicolumn{1}{c|}{Padding} & \multicolumn{1}{c}{ReLU} \\
		\midrule 
		conv1 & 3 & 8 & 3 & 1 & 0 & \ding{51} \\
		conv2 & 8 & 16 & 3 & 1 & 0 & \ding{51} \\
		conv3 & 16 & 32 & 3 & 1 & 0 & \ding{51} \\
		conv4 & 32 & 32 & 3 & 1 & 0 & \ding{51} \\
		max pooling & \textbackslash & \textbackslash & 2 & 2 & 0 & \ding{55} \\
		fc1 & 1152 & 64 & \textbackslash & \textbackslash & \textbackslash & \ding{51} \\
		fc2 & 64 & 2 & \textbackslash & \textbackslash & \textbackslash & \ding{55} \\
		\midrule
		\multicolumn{7}{c}{reglayer: normalization + dropout}
		\\
		\bottomrule
	\end{tabular}
\end{table}

In this section, $c$, $s$, $i$, and $j$ denote the channel index, kernel index, row index, and column index, respectively. Their specific ranges depend on the corresponding layer.
The input to conv1 is $X^{(1)}_{c, i, j}$ ($c \in \{1, 2, 3\}, i \in \{1, \dots, 128\}, j \in \{1, \dots, 128\}$), $X^{(1)}_{c, i, j}$ denotes the value at the $i$-th row and the $j$-th column of the $c$-th channel of the input. The output is
\begin{equation}
	\mathcal{C}^{(1)}_{s,i,j} = \sum_{c=1}^{3} \sum_{u=0}^{2} \sum_{v=0}^{2} W^{(1)}_{s,c,u,v} \cdot X_{c,i+u,j+v} + b^{(1)}_k, \label{eq1}
\end{equation}
where $W^{(1)}_{s,c,u,v}$ represents the weight at position $(u,v)$ of the $s$-th convolutional kernel corresponding to the $c$-th input channel. $b^{(1)}_k$ is the bias of the $s$-th convolutional kernel. After processing by the convolutional kernels, the number of output channels is equal to the number of convolutional kernels (e.g. $s \in \{1, \dots, 8\}, i \in \{1, \dots, 126\}, j \in \{1, \dots, 126\}$).
The ReLU activation is
\begin{equation}
	\mathcal{R}^{(1)}_{s,i,j} = \max(0, \mathcal{C}^{(1)}_{s,i,j}),\label{eq2}
\end{equation}
the max pooling layer is
\begin{equation}
	\mathcal{P}^{(1)}_{s,i,j} = \max_{\substack{u \in \{0,1\} \\ v \in \{0,1\}}} \mathcal{R}^{(1)}_{s,2i-1+u,2j-1+v}, \label{eq3}
\end{equation}
in $\mathcal{P}^{(1)}_{s,i,j}$, $i$ and $j$ range from $1$ to $63$, respectively (e.g., $i \in \{1, \dots, 63\}$ and $j \in \{1, \dots, 63\}$).

The reglayer consists of two parts: batch normalization and dropout. The batch normalization process itself includes two components: normalization and the scale-shift transformation. Here, normalization refers to computing the mean and variance independently for each kernel $s$:
\begin{equation}
	\begin{aligned}
		\mu_s &= \frac{1}{N \cdot h' \cdot w'} \sum_{n,i,j} \mathcal{P}^{(1)}_{n,s,i,j}, \\
		\sigma^2_s &= \frac{1}{N \cdot h' \cdot w'} \sum_{n,i,j} (\mathcal{P}^{(1)}_{n,s,i,j} - \mu_s)^2, \label{eq4}
	\end{aligned}
\end{equation}
where $\mu_s$ and $\sigma^2_s$ are the mean and variance of the $s$-th channel, respectively. $N$ denotes the number of samples, and $h'$ and $w'$ denote the height and width of the feature map, respectively. Here, $n$ indexes the sample dimension, while $i$ and $j$ index the row and column positions within the feature map, respectively.

The scale-shift transformation is as follows:
\begin{equation}
	\mathcal{B}^{(1)}_{s,i,j} = \gamma_s^{(1)} \cdot \frac{\mathcal{P}^{(1)}_{n,s,i,j} - \mu_s}{\sqrt{\sigma^2_k + \epsilon}} + \xi_s^{(1)}, \label{eq5}
\end{equation}
where $\gamma_s^{(1)}$ and $\xi_s^{(1)}$ are learnable scalar parameters for the $s$-th channel, shared across all spatial locations and samples within the batch, and $\epsilon = 10^{-5}$ is a small constant for numerical stability. At the same time, $\mathcal{B}^{(1)}_{s,i,j}$ is randomly set to zero with probability $p$, and scaled accordingly, as follows:
\begin{equation}
	\mathcal{D}^{(1)}_{s,i,j} = 
	\begin{cases}
		\frac{B^{(1)}_{s,i,j}}{1 - p} & {\rm with \ probability} \ 1-p \\
		0 & {\rm with \ probability} \ p, \label{eq6}
	\end{cases}
\end{equation}
the convolution operations in the second, third, and fourth layers are defined in a similar manner.

The output of the fourth convolutional layer is $\mathcal{D}^{(4)} \in \mathbb{R}^{32 \times 6 \times 6}$. After flattening, it is represented as $d_f \in \mathbb{R}^{1152}$.

The output of the first fully connected layer (fc1) is:
\begin{equation}
	z_{i'} = \sum_{f=1}^{1152} W^{(fc1)}_{i',f} \cdot d_f + b^{(fc1)}_{i'}, \label{eq7}
\end{equation}
where $i' \in \{1,\dots,64\}$ indexes the output neurons of the fc1, and $b^{(fc1)}_{i'}$ is the bias term for the $i'$-th output neuron. The ReLU activation is applied as follows:
\begin{equation}
	a_{i'} = \max(0, z_{i'}). \label{eq8}
\end{equation}
The output of the second fully connected layer (fc2) is:
\begin{equation}
	\hat{y}_{j'} = \sum_{i'=1}^{64} W^{(fc2)}_{j',i'} \cdot a_{i'} + b^{(fc2)}_{j'}, \label{eq9}
\end{equation}
where $j' \in \{1, 2\}$ indexes the output neurons of the fc2, and $b^{(fc2)}_{j'}$ is the bias term for the $j'$-th output neuron. 

\subsection{Optimization Problem}
Based on the previously introduced neural network model, we reframe the network training process as a solvable optimization problem in which the parameters to be optimized consist of the weights across all layers, with the optimization objective directed toward minimizing localization error. Central to this reformulation is the definition of the following objective function intended for minimization:
\begin{equation}
	\min_{\substack{x}} f(x) = \frac{1}{N} \sum_{n=1}^{N} \sum_{j'=1}^{2} \left( \hat{y}_{n,j'} - y_{n,j'} \right)^2, \label{eq10}
\end{equation}
where $N$ is the number of samples, $\hat{y}_{i,n}$ and $y_{i,n}$ denote the predicted and ground truth values for sample $i$ and output dimension $n$, respectively. $x = \{ W^{(1)}, b^{(1)}, \gamma^{(1)}, \xi^{(1)}, \dots, \allowbreak W^{(fc2)}, b^{(fc2)}\}$
denotes the set of all optimization parameters, where $x \in \mathbb{R}^d$.

\section{The Optimization Algorithm}
\label{The Optimization Algorithm}
This section proposes an optimization algorithm, termed Diag-OCP, which stands for Diagonal Hessian approximated OCP-based algorithm. The algorithm integrates gradient and Hessian matrix information based on an OCP method \citep{ref25} and is theoretically proven to achieve a convergence rate comparable to that of the Adam under a rigorous non-convex optimization setting.

The OCP method proposes the following update rules \citep{ref25}:
\begin{equation}
	\begin{aligned}
		&x_{k+1} = x_{k} - \phi_{k}(x_k) \\
		&\phi_{l}(x_k) = (R + \nabla^{2} f(x_{k}))^{-1}[\nabla f(x_{k}) + R \phi_{l-1}(x_{k})] \\
		&\phi_{0}(x_{k}) = (R + \nabla^{2} f(x_{k}))^{-1} \nabla f(x_{k}),
	\end{aligned}
\end{equation}
since the OCP method involves inversion of the Hessian matrix, \cite{ref26} replaced $(R + \nabla^{2} f(x_{k}))^{-1}$ with a tunable matrix $M$ to reduce computational cost. The update rules are as follows:
\begin{equation}
	\begin{aligned}
		&x_{k+1} = x_{k} - \phi_{k}(x_{k}) \\
		&\phi_{l}(x_{k}) = M \nabla f(x_{k}) + (I - M \nabla^{2} f(x_{k})) \phi_{l-1}(x_{k}) \\
		&\phi_{0}(x_{k}) = M \nabla f(x_{k}). \label{eq15}
	\end{aligned}
\end{equation}

However, when applied to large-scale neural network models, \eqref{eq15} incurs high storage and computational costs due to the involvement of the Hessian matrix. When the parameter dimension is $\mathbb{R}^d$, the storage cost of the Hessian matrix is $\mathcal{O}(d^2)$, and the computational cost is also $\mathcal{O}(d^2)$, which is clearly impractical.

Therefore, the following algorithmic framework is introduced, which relies only on the diagonal elements of the Hessian matrix and estimates them using the Hutchinson's method. In this case, the computational complexity is reduced from $\mathcal{O}(d^2)$ to $\mathcal{O}(d)$, and the storage cost is reduced to $\mathcal{O}(d)$. The complete algorithmic framework is presented below:
\begin{comment}
	\begin{equation}
		x_{k+1} = x_{k} - (I - (I - M \hat{D}_{k})^{k+1}) \hat{D}_{k}^{-1} \hat{m}_{k}
	\end{equation}	
\end{comment}
\begin{equation}
	\begin{aligned}
		&x_{k+1} = x_{k} - \phi_{k}(x_{k}) \\
		&\phi_{l}(x_{k}) = M \hat{m}_{k} + (I - M \hat{D}_{k}) \phi_{l-1}(x_{k}) \\
		&\phi_{0}(x_{k}) = M \hat{m}_{k}, \label{eq16}
	\end{aligned}
\end{equation}
where $\hat{m}_{k}$ and $\hat{D}_{k}$ are the bias-corrected terms defined as follows:
\begin{equation}
	\begin{aligned}
		&\hat{m}_{k} = \frac{m_{k}}{1 - {\beta_{1}}^{k}} \\
		&\hat{D}_{k} = \frac{D_{k}}{1 - {\beta_{2}}^{k}}, \label{eq17}
	\end{aligned}
\end{equation}
here, $\beta_{1}$ and $\beta_{2}$ are tunable parameters. The exponential moving average terms, $m_{k}$ and $D_{k}$, are defined as follows:
\begin{equation}
	\begin{aligned}
		&m_{k} = \beta_{1} m_{k-1} + (1-\beta_{1}) g_{k}(x_{k}) \\
		&D_{k} = \beta_{2} D_{k-1} + (1 - \beta_{2}) H_{k}, \label{eq18}
	\end{aligned}
\end{equation}
where $g_{k}(x_k)$ is the stochastic gradient at the $k$-th iteration, and $H_k$ is the Hessian diagonal matrix estimated using the Hutchinson's method. The relationship between $g_{k}$ and the exact gradient is as follows, where the noise originates from mini-batch sampling:
\begin{equation}
	g(x_{k}) = \nabla f(x_{k}) + \zeta_{k}, \label{eq19}
\end{equation}
where $\zeta_k$ denotes gradient noise vector at the $k$-th iteration.

At the same time, the Hutchinson's method also performs estimation based on noisy information, and $H_k$ satisfies the following condition:
\begin{equation}
	H_{k} = \frac{1}{\mathcal{N}} \sum_{m=1}^{\mathcal{N}} v_{m} \odot ((\nabla^{2} f(x_{k}) + \varepsilon_{k}) v_{m}), \label{eq20}
\end{equation}
here, $\varepsilon_{k}$ denotes the random noise matrix at the $k$-th iteration, $v$ is the random vector in the Hutchinson's method, and $\mathcal{N}$ represents the number of samples of $v$.

To prevent overfitting and enhance the generalization ability of the model, we incorporate weight decay as a regularization technique during training. It is important to note that the subsequent theoretical analysis is conducted solely on the Algorithm \eqref{eq16} without the weight decay term, in order to isolate and clarify its convergence properties.
\begin{equation}
	x_k' = x_k(1 - \alpha \lambda)
\end{equation}
where $x_k$ denotes the parameters at iteration $k$, $\alpha$ is the learning rate (LR), and $\lambda$ is the weight decay coefficient.

The implementation procedure of the above algorithm is illustrated in the following pseudocode:
\begin{algorithm}[H]
	\caption{Diag-OCP \eqref{eq16}}\label{alg:alg1}
	\begin{algorithmic}
		\STATE 
		\STATE {\textsc{Initialization}}
		\STATE \hspace{0.5cm}Initialize matrix $M \gets \alpha I$
		\STATE \hspace{0.5cm}Initialization of parameters $\beta_1$ and $\beta_2$
		\STATE \hspace{0.5cm}Initialization of $m_0$, $D_0$ and $v_{m}$
		\STATE \hspace{0.5cm}Initialization of $x_0, \mathcal{N}$
		\STATE \hspace{0.5cm}$ k \gets 0 $
		\STATE 
		\STATE {\textsc{Algorithm Procedure}}
		\STATE \hspace{0.5cm}While $k < T$ do
		\STATE \hspace{0.5cm}\hspace{0.5cm}$\begin{aligned}
			H_{k} = \frac{1}{\mathcal{N}} \sum_{m=1}^{\mathcal{N}} v_{m} \odot ((\nabla^{2} f(x_{k}) + \varepsilon_{k}) v_{m})
		\end{aligned}$
		\STATE \hspace{0.5cm}\hspace{0.5cm}$m_{k} = \beta_{1} m_{k-1} + (1-\beta_{1}) g_{k}(x_{k})$
		\STATE \hspace{0.5cm}\hspace{0.5cm}$[H_k]_{ii} = \max([H_k]_{ii}, \mu)$
		\STATE \hspace{0.5cm}\hspace{0.5cm}$D_{k} = \beta_{2} D_{k-1} + (1 - \beta_{2}) H_{k}$
		\STATE \hspace{0.5cm}\hspace{0.5cm}$\hat{m}_{k} = \frac{m_{k}}{1 - {\beta_{1}}^{k}}$
		\STATE \hspace{0.5cm}\hspace{0.5cm}$\hat{D}_{k} = \frac{D_{k}}{1 - {\beta_{2}}^{k}}$
		\STATE \hspace{0.5cm}\hspace{0.5cm}$\phi_{0}(x_{k}) = M \hat{m}_k$
		\STATE \hspace{0.5cm}\hspace{0.5cm}$l \gets 0$
		\STATE \hspace{0.5cm}\hspace{0.5cm}\hspace{0.5cm}While $l < k$ do
		\STATE \hspace{0.5cm}\hspace{0.5cm}\hspace{0.5cm}\hspace{0.5cm}$l \gets l+1$
		\STATE \hspace{0.5cm}\hspace{0.5cm}\hspace{0.5cm}\hspace{0.5cm}$\begin{aligned}
			\phi_{l}(x_{k}) &= M \hat{m}_k + (I - M \hat{D}_k) \phi_{l-1}(x_{k})
		\end{aligned}$
		\STATE \hspace{0.5cm}\hspace{0.5cm}\hspace{0.5cm}Return $\phi_{k}(x_{k})$
		\STATE \hspace{0.5cm}\hspace{0.5cm}$x_k' = x_k(1 - \alpha \lambda)$
		\STATE \hspace{0.5cm}\hspace{0.5cm}$x_{k+1} = x_k' - \phi_{k}(x_{k})$
		\STATE \hspace{0.5cm}\hspace{0.5cm}$k \gets k+1$
		\STATE \hspace{0.5cm}Return $x_{T}$
	\end{algorithmic}
	\label{alg1}
\end{algorithm}

\begin{comment}
	Since $\varepsilon_k$ originates from the noise in stochastic gradient computation and $v$ is an independently sampled vector used in the Hutchinson method, the noise matrix $\varepsilon_k$ and the Hutchinson vector $v$ are generally independent given $x_k$.
\end{comment}

\begin{comment}
	\textbf{Lemma 2.} Under the assumptions that $E[\varepsilon_k v | x_k] = E[\varepsilon_k | x_k] E[v | x_k]$ and $\mathbb{E} [ \varepsilon_k | x_k ] = 0$ hold, it follows that:
	\begin{equation}
		E[H_k | x_k] = {\rm diag} (\nabla^2 f(x_k)). \label{eq21}
	\end{equation}
	The proof of unbiasedness can be found in Appendix A.
	
	In the following, we formalize the assumptions required in our convergence analysis.
\end{comment}
To demonstrate the convergence rate of the Diag-OCP, we propose the following reasonable assumptions.
\begin{assumption} \label{ass:main}
	The following assumptions hold throughout our analysis:
	\begin{enumerate}[label=(A\arabic*), leftmargin=*, nosep]
		\item \label{ass:a1} (Smoothness and Boundedness) The objective function $f$ is differentiable and has $L$-Lipschitz continuous gradient, i.e., $\forall x, y \in \mathbb{R}^d, \exists L > 0: \|\nabla f(x) - \nabla f(y)\| \leq L\|x - y\|$. Moreover, $f$ is lower bounded: $\exists f^* > -\infty$ such that $f(x) \geq f^*$.
		
		\item \label{ass:a2} (Bounded Gradients and Noise) The true gradients and stochastic gradients are bounded, i.e., $\|\nabla f(x)\| \leq \mathcal{H}$, $\|g_{k}(x_k)\| \leq \mathcal{H}_g$, and the gradient noise has bounded variance: $\mathbb{E}[\|\zeta_k\|^2] \leq \sigma_g^2$. Here, $\mathcal{H}$, $\mathcal{H}_g$, and $\sigma_g^2 > 0$ are constants.
		
		\item \label{ass:a3} (Unbiased Noise) The gradient and Hessian estimation noises are unbiased, i.e., $\mathbb{E}[\zeta_k | x_k] = 0$ and $\mathbb{E}[\varepsilon_k | x_k] = 0$. Furthermore, both $\zeta_k$ and $\varepsilon_k$ are independent of the historical information $\mathcal{F}_{k-1} = \sigma(\{g_i, H_i\}_{i=1}^{k-1})$.
		
		\item \label{ass:a4} (Bounded Hessian Approximation) The estimated Hessian matrix is bounded from below and above: $\mu I \preceq H_k \preceq G_d I$ for constants $\mu, G_d > 0$. This is enforced by our diagonal clipping procedure in Algorithm 1.
		
		\item \label{ass:a5} (Stability of Update Matrix) The matrix $I - M\hat{D}_k$ is stable, i.e., $\|I - M\hat{D}_k\| < 1$.
	\end{enumerate}
\end{assumption}
The convergence rate of the Diag-OCP is established under Assumption \ref{ass:main}. It is crucial to address the practical validity of these assumptions within the context of deep learning, particularly concerning the smoothness requirement \ref{ass:a1} and the noise model.

Firstly, although the use of non-smooth activation functions such as ReLU appears to contradict the twice continuously differentiable requirement in \ref{ass:a1}, our analysis focuses on the effective behavior of the optimization process. In practice, due to floating-point precision limitations and random initialization, the input to a ReLU neuron is almost never exactly zero, rendering the loss function differentiable almost everywhere with respect to the parameters. Moreover, modern deep learning frameworks define a subgradient at zero, ensuring that the computation remains well-posed. Therefore, the smoothness assumption provides a valid theoretical framework for analyzing the algorithm's performance in practical scenarios. Experimental results on standard datasets further confirm that the algorithm converges reliably in practice.

Secondly, the stochastic noise models for the gradient and Hessian in \eqref{eq19} and \eqref{eq20} are not arbitrary additions but precise abstractions of the algorithm's inherent stochasticity. The gradient noise $\zeta_k$ arises naturally from mini-batch sampling during training. Similarly, the Hessian noise $\varepsilon_k$ originates from the Hutchinson estimator's reliance on random vectors and the underlying stochastic gradients. Our assumptions on this noise (unbiasedness and bounded variance in \ref{ass:a2}, \ref{ass:a3}) are standard and mild, reflecting the typical environment of stochastic optimization. The convergence rate under these conditions demonstrates the algorithm's robustness to the noise sources present in actual deep learning training.
\begin{comment}
	\textbf{Assumption}
	
	1: $f$ is differentiable and has $L$-Lipschitz gradient, i.e. $\forall x, y \in \mathbb{R}^d, \exists L > 0 : \| \nabla f(x) - \nabla f(y) \| \leq L \| x - y \|$. Moreover, $f$ is lower bounded: $\exists f^* > -\infty$ such that $f(x) \geq f^*$.
	
	2: Gradients and noise are bounded, i.e. $\| \nabla f(x) \| \leq \mathcal{H}, \| g(x_k) \| \leq \mathcal{H}_g, \mathbb{E} [ \| \zeta_k \|^2 ] \leq \sigma_g^2$. Here, $\mathcal{H}$, $\mathcal{H}_g$, and $\sigma_g^2 > 0$ are constants.
	
	3: The noise is unbiased, i.e. $\mathbb{E} [ \zeta_k | x_k ] = 0, \mathbb{E} [ \varepsilon_k | x_k ] = 0$. Both $\zeta_k$ and $\varepsilon_k$ are independent of the historical information $\mathcal{F}_{k-1} = \sigma (\{ g_i, H_i \}_{i=1}^{k-1})$.
	
	4: The estimated Hessian matrix satisfies $\mu I \preceq H_k \preceq G_d I$ with $\mu$ and $G_d > 0$, 
	which is enforced by our diagonal clipping procedure in Algorithm 1.
	
	5: $\|I - M \hat{D}_k\| < 1$.
\end{comment}

Based on the reasonable Assumption \ref{ass:main} above, we have the following theorem:
\begin{theorem} \label{thm:convergence}
	Suppose that Assumption \ref{ass:main} holds. By choosing an appropriate matrix $M$ and setting the exponential moving average parameters $\beta_1, \beta_2 \in [0, 1)$, we can establish the non-asymptotic convergence guarantee for Diag-OCP \eqref{eq16} (without weight decay). Our analysis shows that the algorithm achieves a convergence rate of $\mathcal{O} ( \frac{1}{T} )$ for any finite $T > 0$. Specifically, it holds that:
	\begin{equation}
		\mathop{\mathrm{min}}\limits_{k \in [1,T]}\mathbb{E} \left[ \| \nabla f(x_k) \|^2 \right] \leq \mathcal{O} ( \frac{1}{T} )
	\end{equation}
\end{theorem}

The proof of the Theorem \ref{thm:convergence} is provided in detail in Appendix.

\begin{remark}
	Although both the proposed algorithm and Adam exhibit sublinear convergence properties, there exists an essential difference in their convergence rates. The proposed algorithm achieves a convergence rate of $\mathcal{O} ( \frac{1}{T} )$, which is significantly better than Adam's rate of $\mathcal{O} (\frac{\log T}{\sqrt{T}})$. Theoretical analysis demonstrates that the proposed algorithm attains a faster convergence speed and superior asymptotic performance when the number of iterations is sufficiently large.
\end{remark}

\section{Simulation}
\label{Simulation}
In the simulation validation section, the publicly available KITTI dataset \citep{ref29} was utilized to evaluate the proposed algorithm against five widely used deep learning-based algorithms. All experiments are designed under resource-limited constraints, employing lightweight networks and a limited number of training iterations to simulate edge learning scenarios. Furthermore, we incorporate weight decay during the parameter update step to enhance the practical performance of the algorithm, emphasizing that this practical modification is not included in the theoretical analysis.
\subsection{Experimental Setup}
The parameters of the Diag-OCP are set as follows: parameters $\beta_{1}$, $\beta_{2}$, $m_{0}$, $D_{0}$, and $\mathcal{N}$ are assigned values of 0.9, 0.999, 0, 0, and 1, respectively. All convolutional and fully-connected layers are initialized using Kaiming uniform initialization, which is suitable for ReLU activations. In the dropout layer, the dropout probability $p$ is set to 0.3, which is an empirically chosen value. Batch normalization layers are initialized with scale parameter $\gamma=1$ and shift parameter $\xi=0$. The elements of the random vector $v$ are drawn from a standard normal distribution. A fixed random seed is set to ensure reproducibility of the initial weights. The weight decay coefficient $\lambda$ is set to 0.008. To facilitate parameter tuning, a diagonal matrix $M$ with identical diagonal entries, denoted by $\alpha I$ is employed. The selection of parameter $\alpha$ is detailed in Section \ref{Learning Rate Tuning Strategy}.

All experiments were conducted in an environment utilizing PyTorch 1.12.1 (with CUDA 11.6) and Python 3.7.12. The hardware platform featured an NVIDIA GeForce GTX 1650 GPU (7GB VRAM) and an AMD Ryzen 5 4600H CPU, under the Windows 10 operating system.

\subsection{Learning Rate Tuning Strategy}
\label{Learning Rate Tuning Strategy}
In this study, we adopted a staged hybrid interval sampling strategy to search for an appropriate LR. This strategy first performed a coarse search on a logarithmic scale to cover critical orders of magnitude (e.g., from $10^{-1}$ to $10^{-4}$), which enabled the rapid localization of the effective learning rate range. Subsequently, we conducted geometric subdivisions within the promising order of magnitude (e.g., subdividing around $10^{-3}$ into $1 \times 10^{-3}$, $5 \times 10^{-4}$, and $1 \times 10^{-4}$). This approach combined the breadth of logarithmic search with the precision of geometric subdivision, thereby avoiding the risk of missing important intervals inherent in purely logarithmic sampling, while also overcoming the inefficiency of linear intervals within small learning rate ranges. Experimental results demonstrated that this strategy efficiently balanced search range and tuning precision, making it suitable for the empirical optimization of learning rates in deep neural networks. The specific process and results of the step size adjustment are presented in Table \ref{tab:alg_results} and Figure \ref{fig:all_algorithms} below.
\subsection{Ablation Study on Clipping Threshold}
To investigate the impact of enforcing a lower bound on the estimated diagonal Hessian matrix, we conducted an ablation study by clamping negative or near-zero values to three different thresholds: 0.001, 0.0001, and 0.00001. The clamping operation ensures the diagonal curvature matrix remains positive definite, thereby avoiding numerical instability in the inverse scaling. All experiments were conducted using a fixed LR of 0.05.

The results are shown in Figure \ref{fig_mu}, where solid lines denote training loss and dashed lines represent validation loss. Each color corresponds to a different clamping threshold, along with the original (unclamped) baseline. As observed, the training and validation curves for the three thresholds are nearly identical and closely match the performance of the unclamped baseline. Notably, clamping appears to slightly accelerate the early-stage empirical loss reduction rate. These findings indicate that imposing a small positive lower bound on the estimated Hessian diagonal does not negatively affect the rate of loss reduction, and may even provide stability benefits during the initial training phase.
\begin{figure}[ht]
	\centering
	\includegraphics[width=3in]{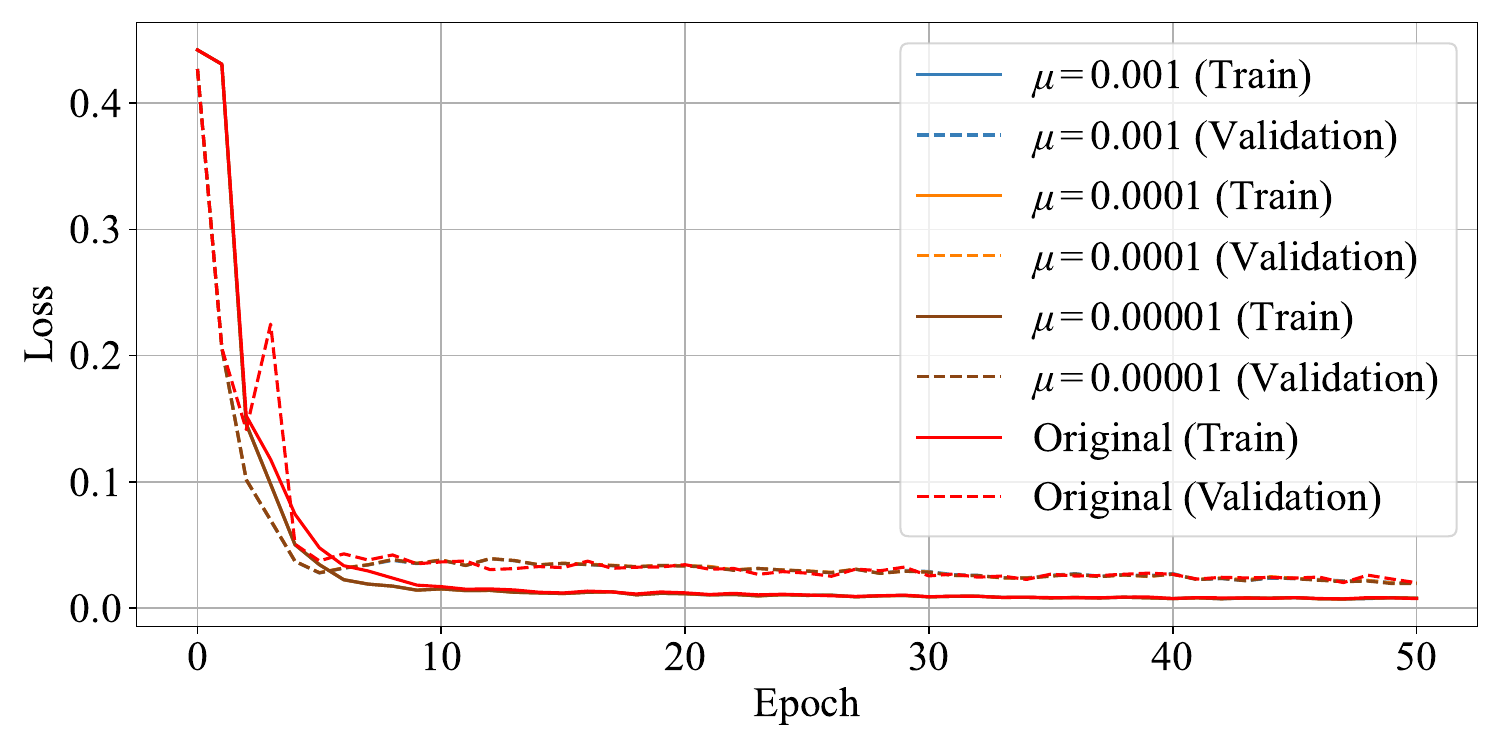}
	\caption{Effect of clamping threshold $\mu$}
	\label{fig_mu}
\end{figure}

Although the training and validation curves under different clamping thresholds (0.001, 0.0001, and 0.00001) were largely similar, the setting $\mu$ = 0.0001 consistently achieved the lowest training and validation losses. Based on this observation, we selected $\mu$ = 0.0001 as the default lower bound for the diagonal Hessian approximation in the following experiments.
\subsection{Comparative Analysis of Optimization Algorithms}
As shown in Figure \ref{fig_heatmap}, the heatmap compares the validation loss of different algorithms after 50 iterations across a range of LRs. Diag-OCP maintains a low loss (0.0135–0.0285) within the 0.005 to 0.05 interval, with a uniform color distribution, indicating its insensitivity to hyperparameter settings. In contrast, AdaHessian exhibits a sharp increase in loss at LR = 0.01, followed by a sudden decrease at LR = 0.05, demonstrating its instability under large learning rates.
\begin{figure}[ht]
	\centering
	\includegraphics[width=3in]{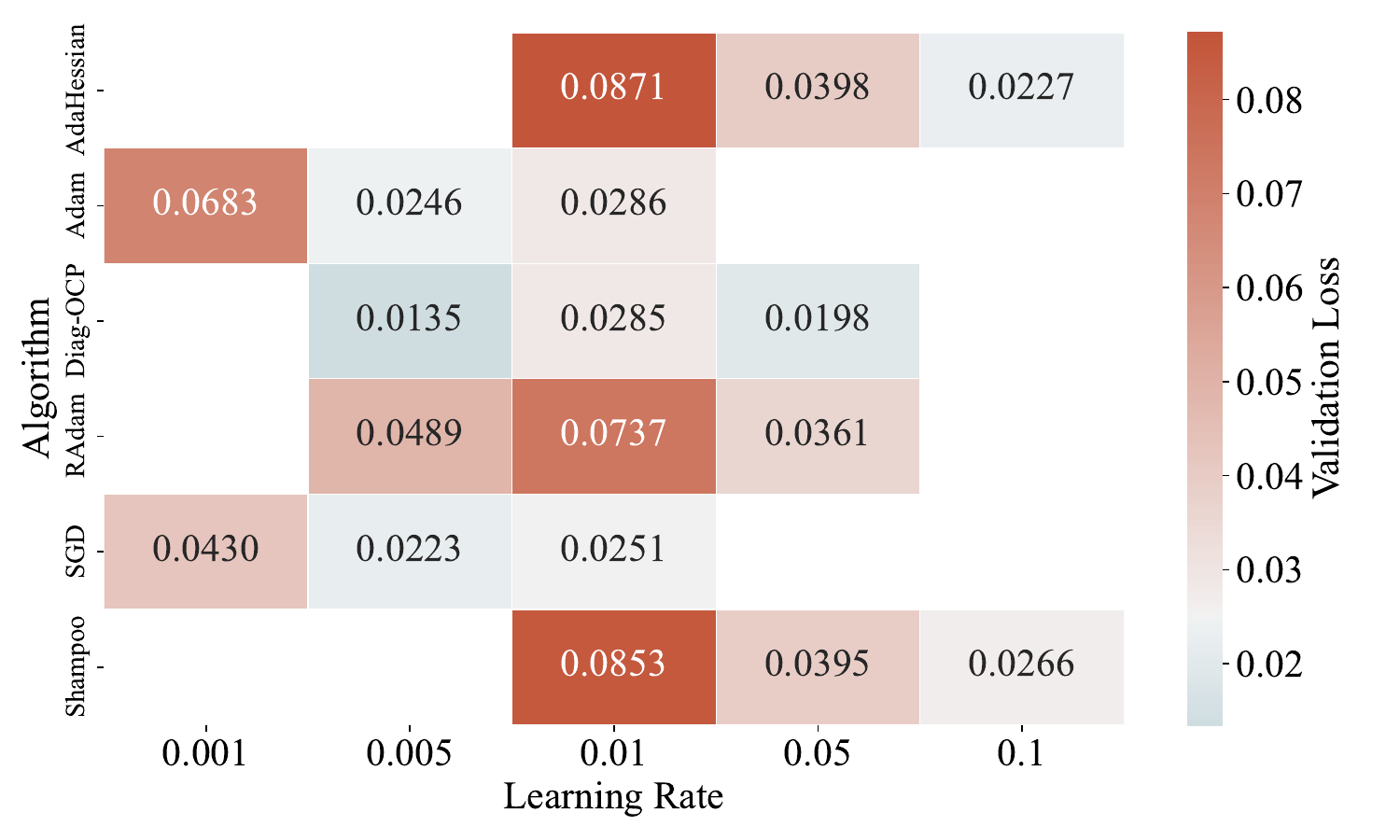}
	\caption{Parameter sensitivity}
	\label{fig_heatmap}
\end{figure}

\begin{comment}
	\begin{table*}[h]
		\centering
		\caption{Comparison Results of The Five Algorithms}
		\begin{tabular*}{0.78\linewidth}{ c | c | c | c | c  }
			\toprule
			Algorithm & \multicolumn{1}{c|}{Tuned parameter} & Train loss & Validation loss & Minimum validation loss (current iteration step) \\
			\midrule 
			Adam & 0.005 & 0.0094 & 0.0246 & 0.0220 (44) \\ 
			Diag-OCP & 0.05 & 0.0076 & 0.0208 & 0.0182 (47) \\
			AdaHessian & 0.1 & 0.0078 & 0.0227 & 0.0227 (50) \\
			RAdam & 0.05 & 0.0095 & 0.0361 & 0.0276 (43) \\
			SGD & 0.005 & 0.0098 & 0.0223 & 0.0223 (50) \\
			Shampoo & 0.1 & 0.0147 & 0.0266 & 0.0257 (30) \\
			\bottomrule
		\end{tabular*}
	\end{table*}
\end{comment}

\begin{table*}[h]
	\centering
	\caption{Comparison results of the six algorithms at 50 and 150 iteration steps}
	\label{tab:alg_results}
	\begin{tabular*}{0.85\linewidth}{ c | c | c | c | c | c }
		\toprule
		Algorithm & Iteration step & Tuned parameter & Train loss & Validation loss & Minimum validation loss \\
		\midrule 
		\multirow{2}{*}{Adam} 
		& 50  & 0.005 & 0.0094 & 0.0246 & 0.0220 \\
		& 150 & 0.005 & 0.0063 & 0.0093 & 0.0058 \\ \midrule
		\multirow{2}{*}{Diag-OCP} 
		& 50  & 0.005  & 0.0078 & 0.0139 & 0.0102 \\
		& 150 & 0.005  & 0.0059 & 0.0041 & 0.0029 \\ \midrule
		\multirow{2}{*}{AdaHessian} 
		& 50  & 0.1   & 0.0078 & 0.0227 & 0.0227 \\
		& 150 & 0.1   & 0.0062 & 0.0213 & 0.0117 \\ \midrule
		\multirow{2}{*}{RAdam} 
		& 50  & 0.05  & 0.0095 & 0.0361 & 0.0276 \\
		& 150 & 0.05  & 0.0049 & 0.0176 & 0.0033 \\ \midrule
		\multirow{2}{*}{SGD} 
		& 50  & 0.005 & 0.0098 & 0.0223 & 0.0223 \\
		& 150 & 0.005 & 0.0068 & 0.0151 & 0.0142 \\ \midrule
		\multirow{2}{*}{Shampoo} 
		& 50  & 0.1   & 0.0147 & 0.0266 & 0.0257 \\
		& 150 & 0.1   & 0.0083 & 0.0204 & 0.0113 \\
		\bottomrule
	\end{tabular*}
\end{table*}

Figure \ref{fig:all_algorithms} presents the training processes of six algorithms under different LRs, further illustrating the hyperparameter tuning process and the sensitivity of each algorithm to LR settings.

\begin{figure*}[ht]
	\centering
	% 第一行
	\subfloat[Adam]{%
		\includegraphics[width=0.32\textwidth]{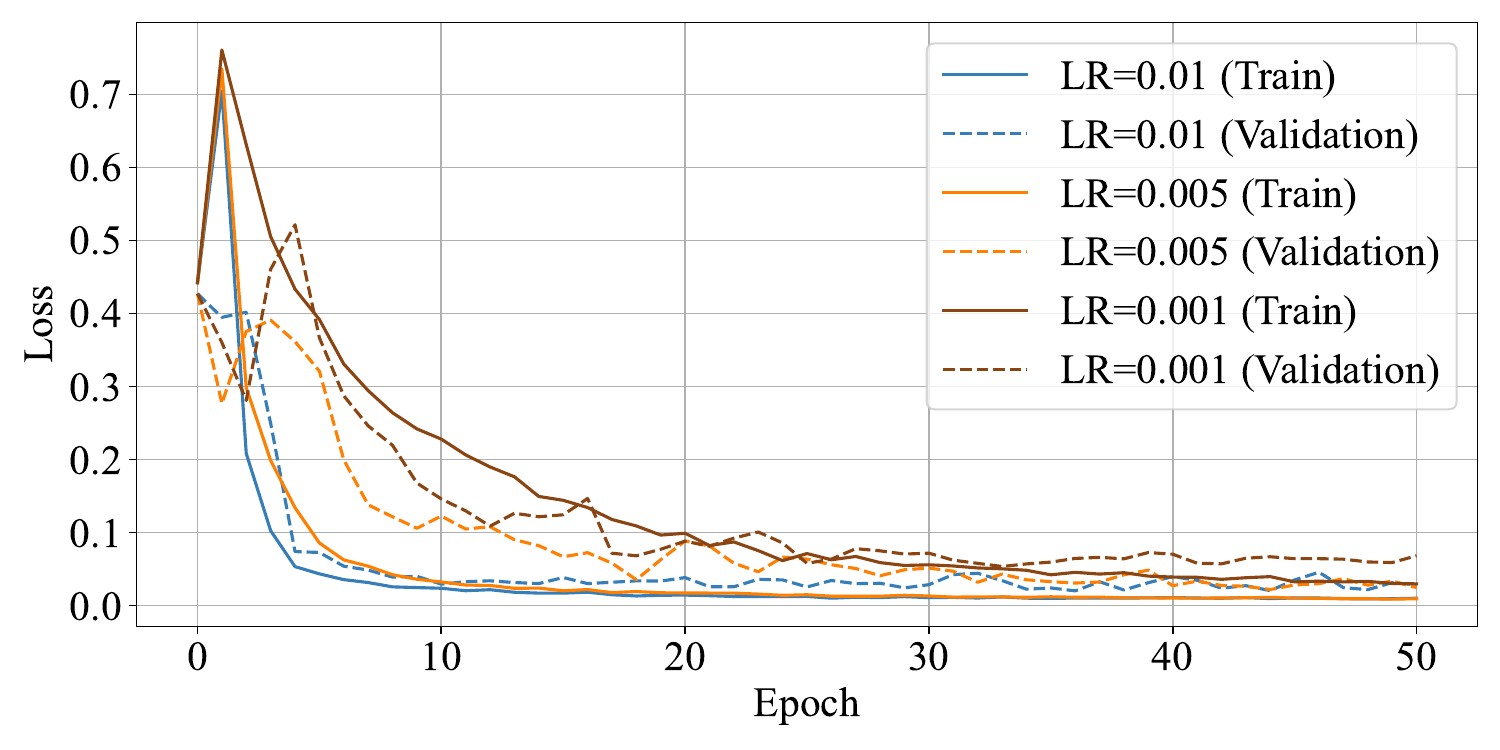}%
		\label{fig_4}}
	\hfill
	\subfloat[Diag-OCP]{%
		\includegraphics[width=0.32\textwidth]{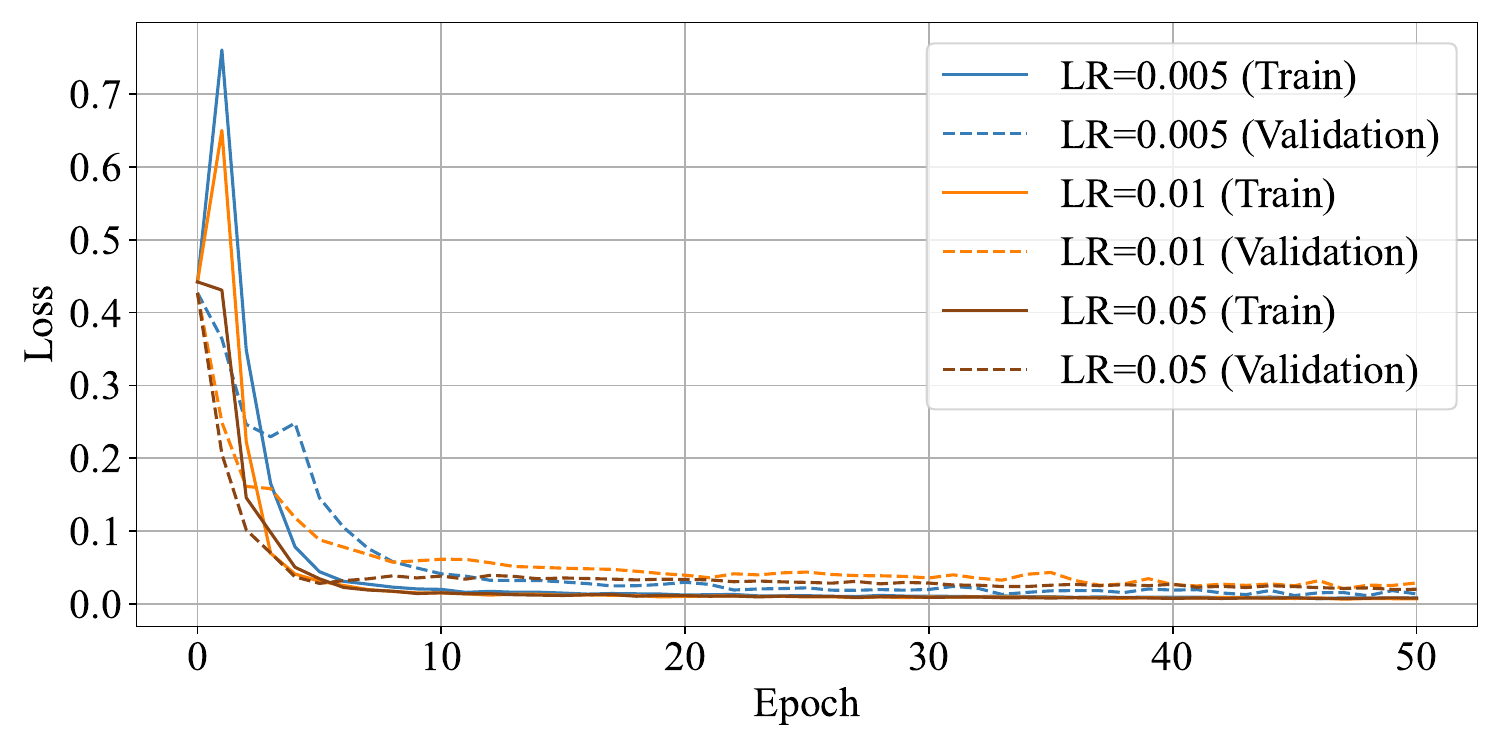}%
		\label{fig_5}}
	\hfill
	\subfloat[AdaHessian]{%
		\includegraphics[width=0.32\textwidth]{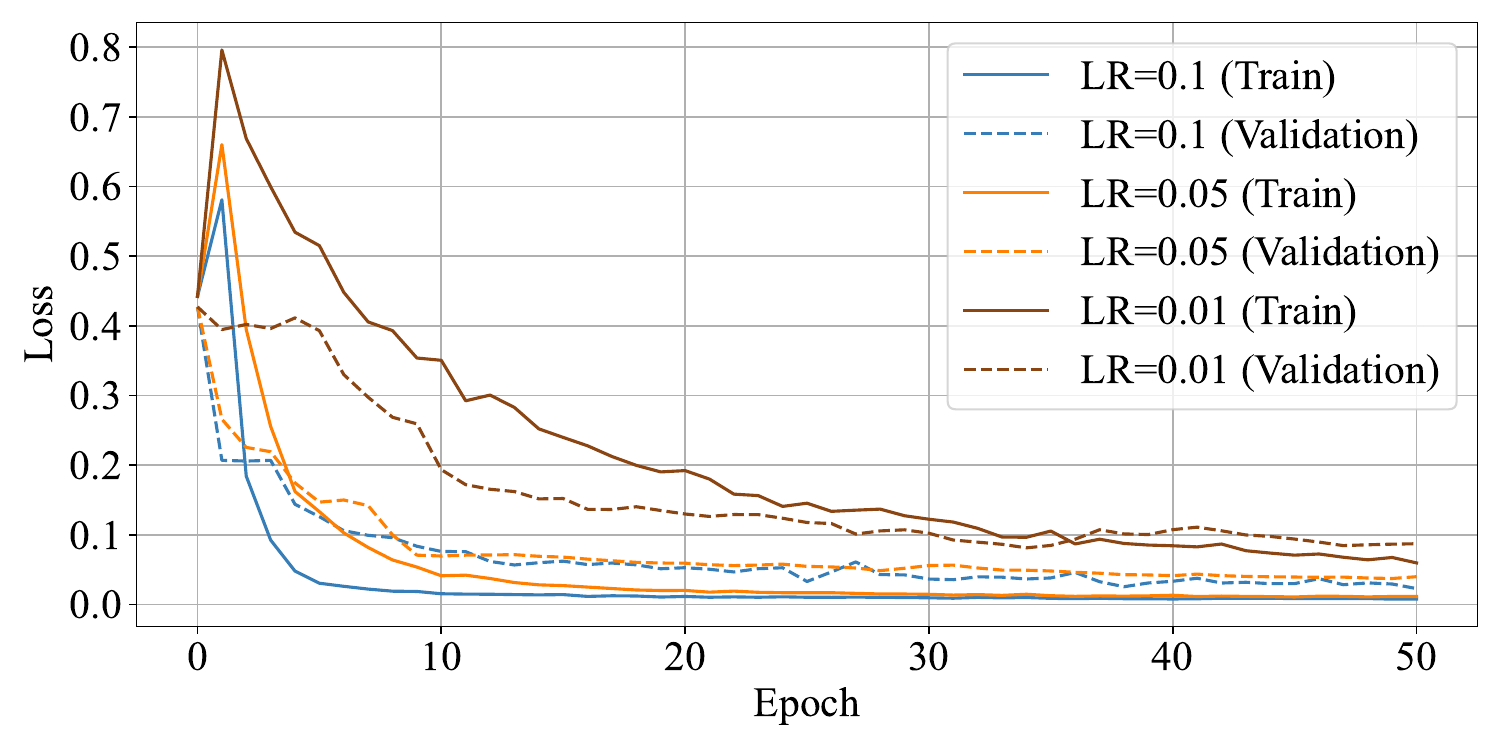}%
		\label{fig_6}}\\[1ex] % <-- 调整两行间距
	
	% 第二行
	\subfloat[RAdam]{%
		\includegraphics[width=0.32\textwidth]{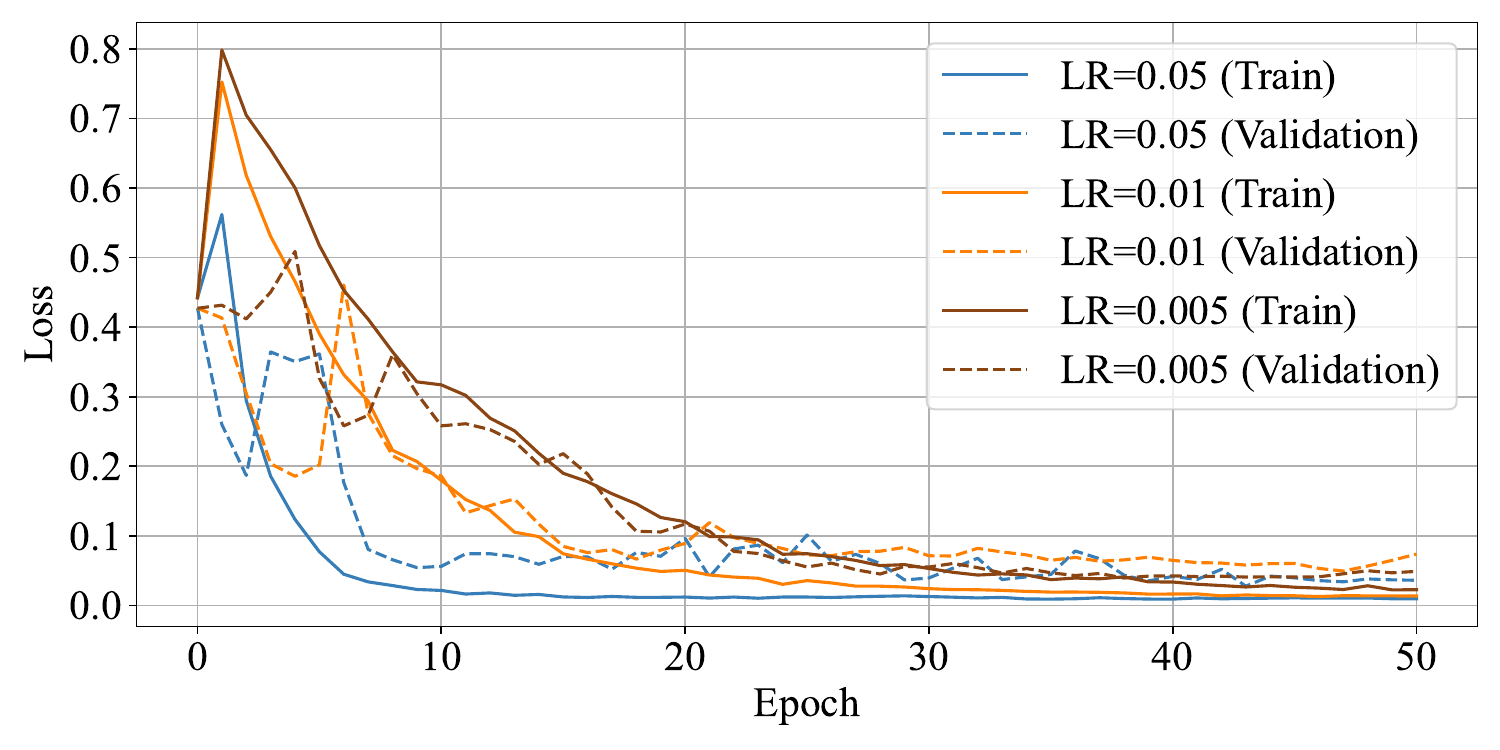}%
		\label{fig_7}}
	\hfill
	\subfloat[SGD]{%
		\includegraphics[width=0.32\textwidth]{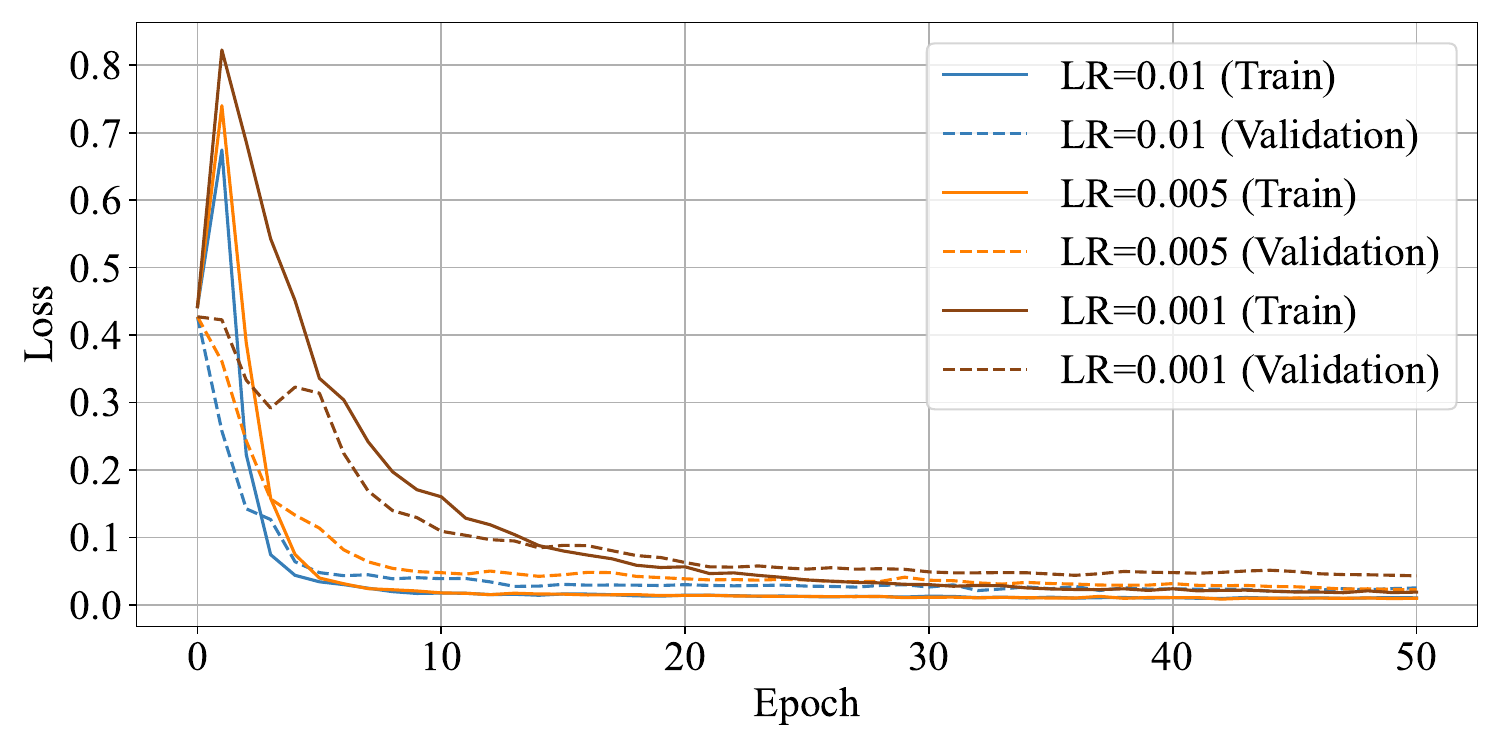}%
		\label{fig_8}}
	\hfill
	\subfloat[Shampoo]{%
		\includegraphics[width=0.32\textwidth]{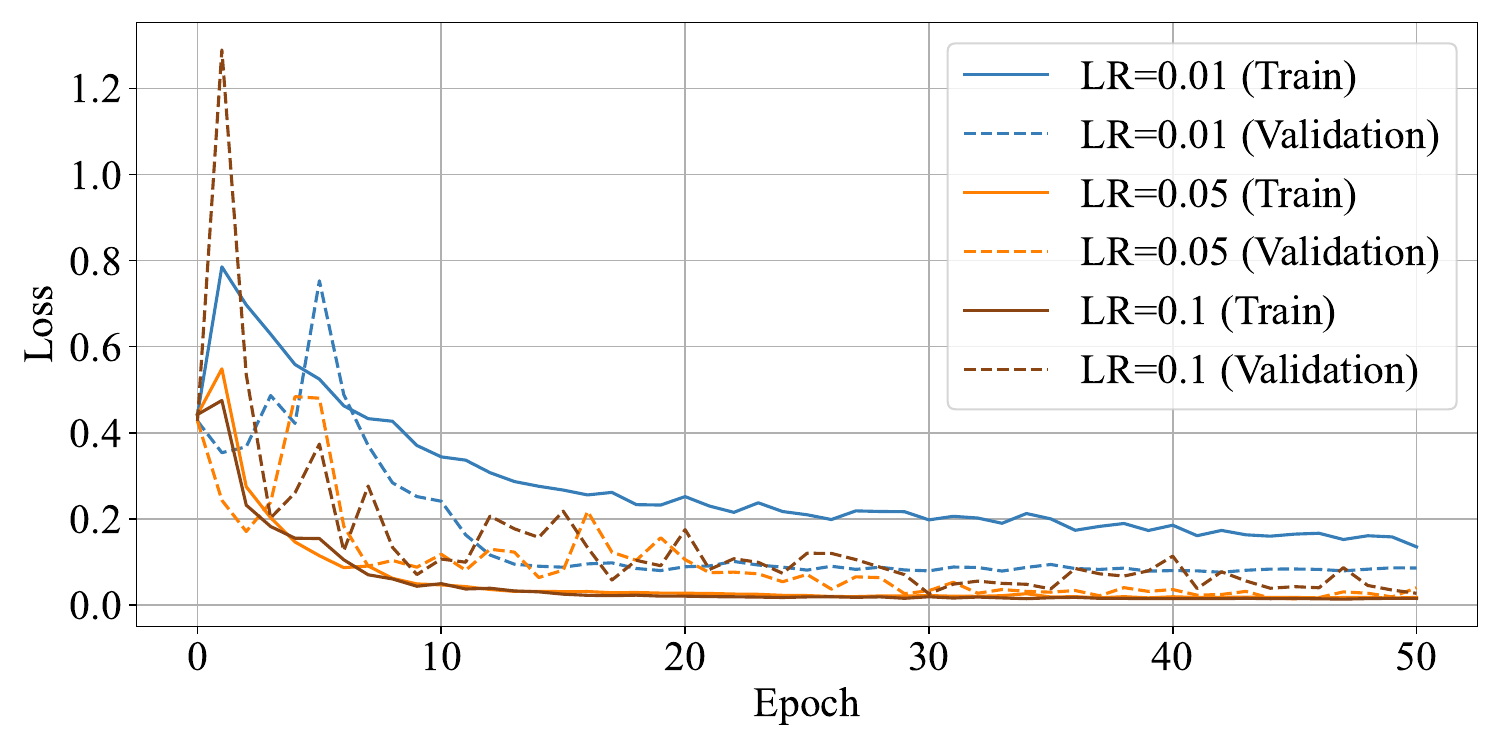}%
		\label{fig_9_plus}}
	
	\caption{Training and validation loss curves of six algorithms under different LR settings: 
		(a) Adam, (b) Diag-OCP, (c) AdaHessian, 
		(d) RAdam, (e) SGD, (f) Shampoo.}
	\label{fig:all_algorithms}
\end{figure*}

Figure \ref{fig:train_val_loss} compares the performance of six algorithms under their respective best-tuned parameter, with training and validation loss presented separately. To further illustrate the stability and performance of each algorithm, the training processes after 50 and 150 iterations are both shown.

\begin{figure*}[h]
	\centering
	% 第一行
	\subfloat[Train loss at 50 iterations]{%
		\includegraphics[width=0.45\textwidth]{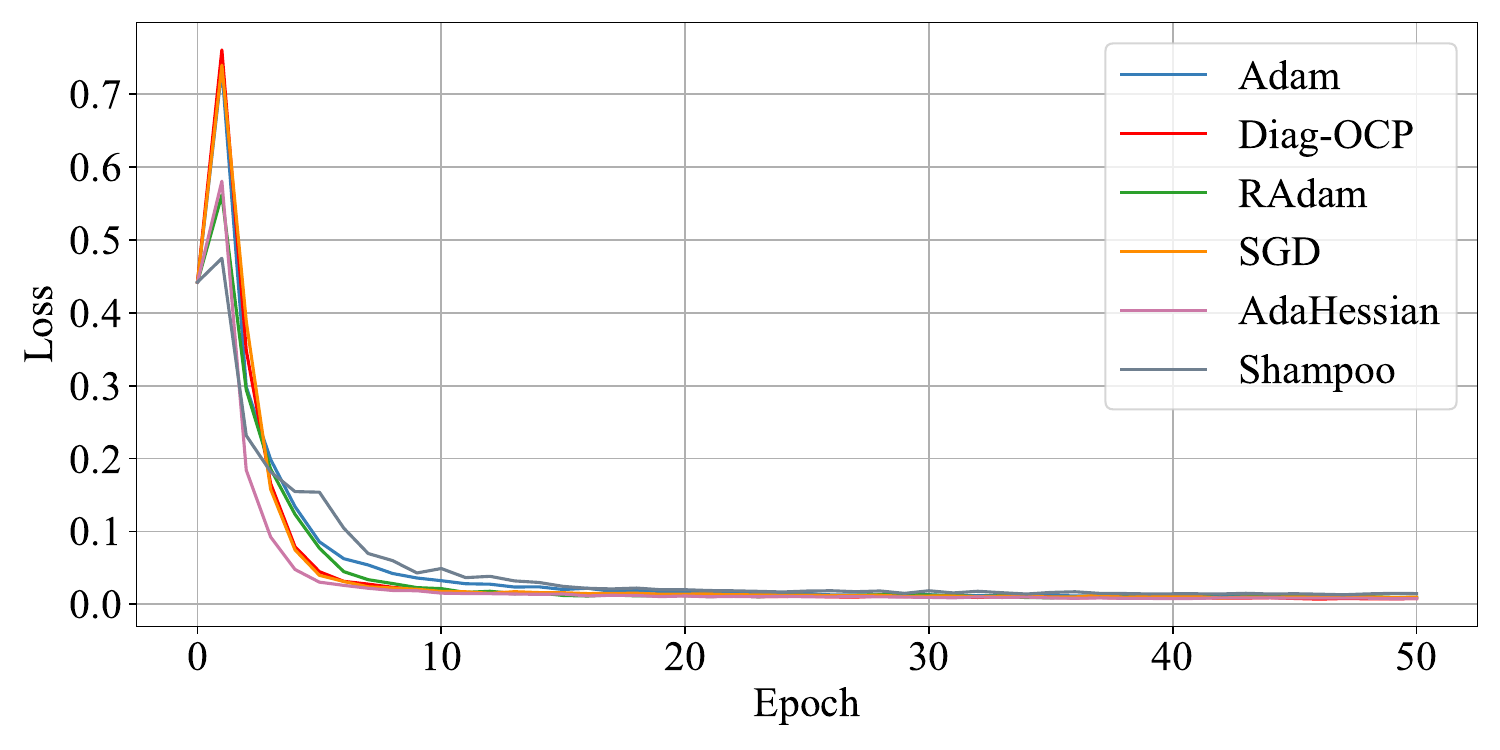}%
		\label{fig_9}}
	\hspace{0.02\textwidth}
	\subfloat[Validation loss at 50 iterations]{%
		\includegraphics[width=0.45\textwidth]{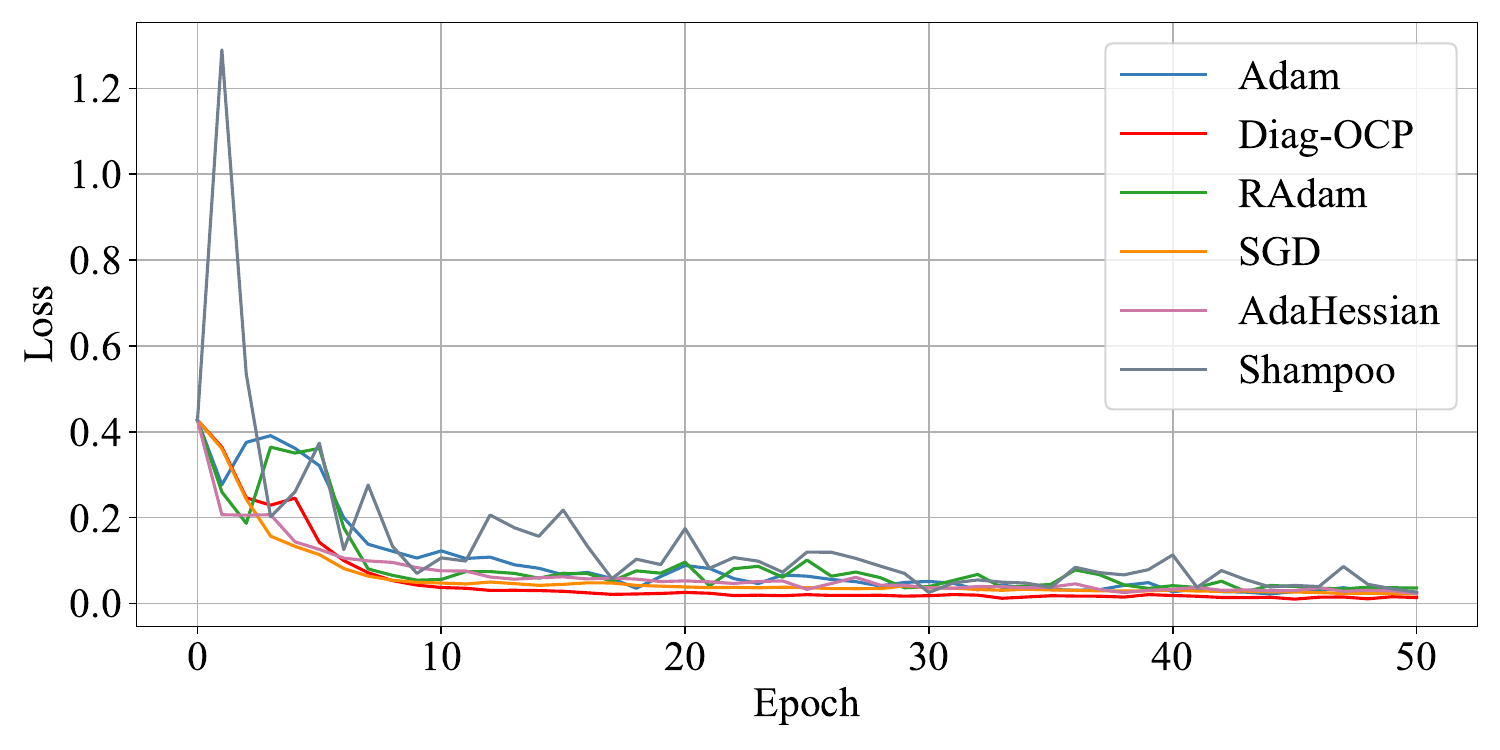}%
		\label{fig_10}}\\[1ex] % 上下行间距
	
	% 第二行
	\subfloat[Train loss at 150 iterations]{%
		\includegraphics[width=0.45\textwidth]{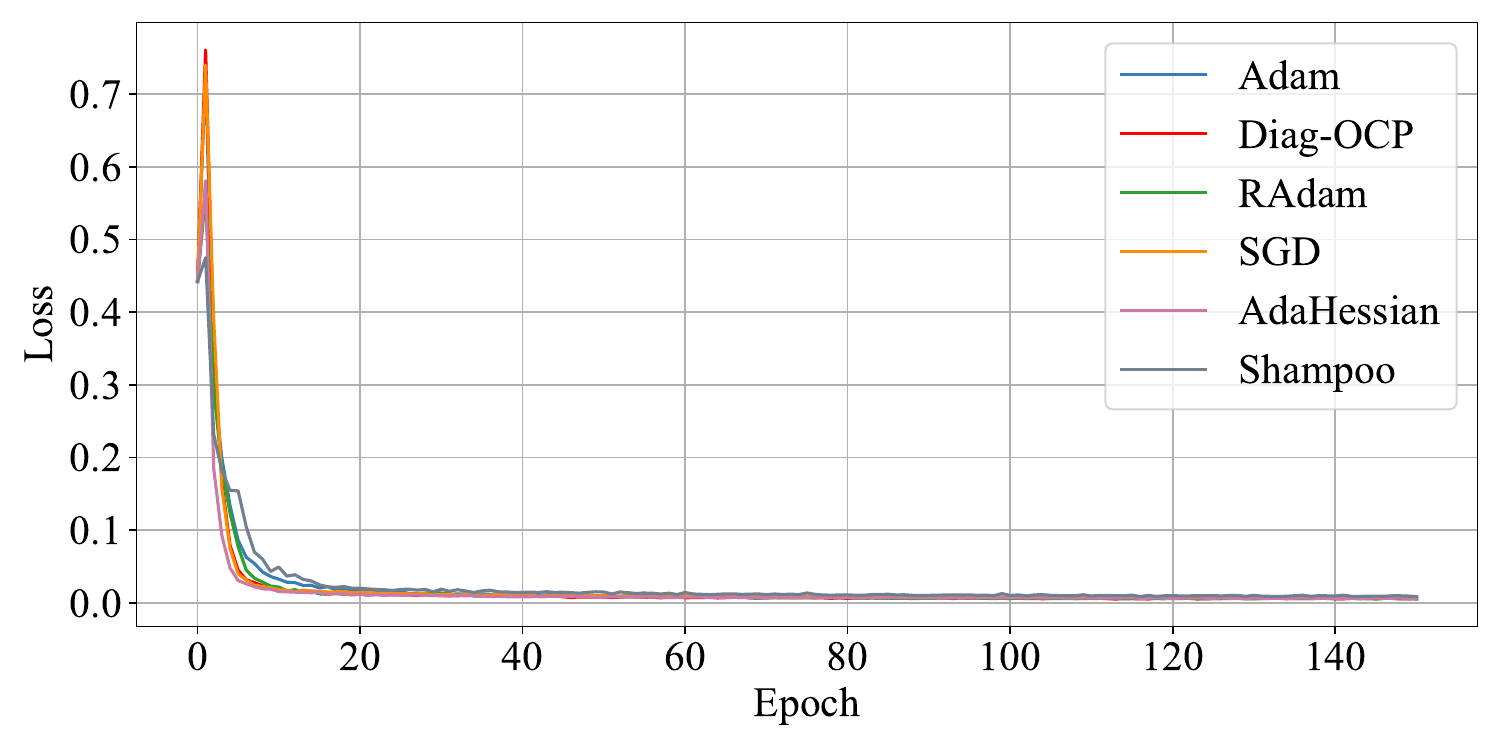}%
		\label{fig_13}}
	\hspace{0.02\textwidth}
	\subfloat[Validation loss at 150 iterations]{%
		\includegraphics[width=0.45\textwidth]{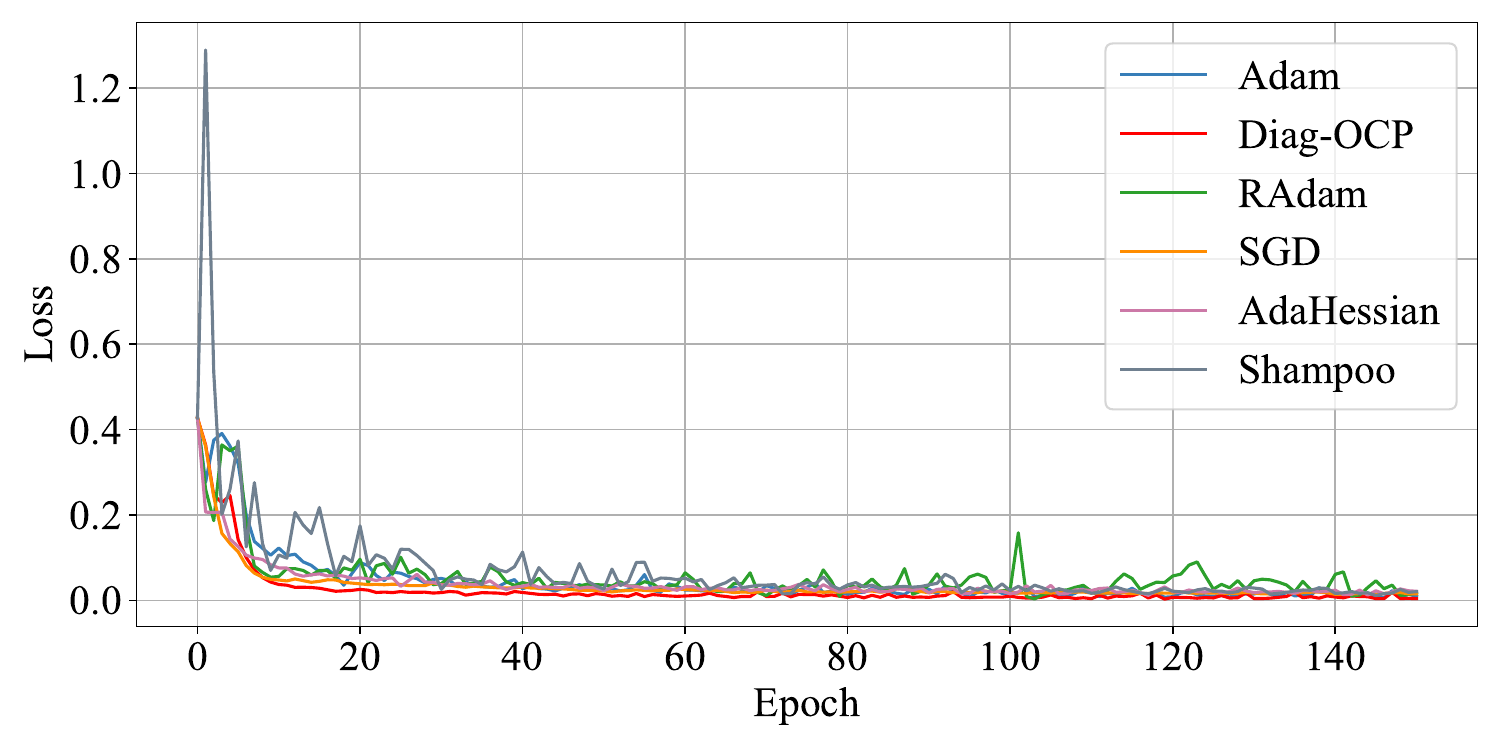}%
		\label{fig_14}}
	
	\caption{Training and validation loss comparisons at different iteration steps}
	\label{fig:train_val_loss}
\end{figure*}

At 50 iterations, Diag-OCP exhibits a rapid decrease in both training and validation loss, demonstrating strong generalization capability. By 150 iterations, Diag-OCP continues to decrease steadily, ultimately achieving the lowest training and validation loss. After 150 iterations, Diag-OCP achieves a validation loss of 0.0041, which is $57.1 \%$ lower than that of the second-best algorithm, Adam. Moreover, the minimum validation loss over the entire training process is 0.0029, representing a $12.1 \%$ improvement over the second-best algorithm, RAdam.

Experimental results indicate that Diag-OCP significantly outperforms the comparison algorithms in terms of parameter sensitivity, validation loss reduction rate, and generalization performance. Its robustness to learning rate selection reduces the difficulty of hyperparameter tuning. Furthermore, due to the high precision of its minimum validation loss, early stopping strategies can be effectively applied in practical scenarios. The training and validation loss curves both exhibit rapid declines during the initial iterations, with the validation loss decreasing smoothly throughout, further confirming the efficiency and stability of the proposed algorithm.

\section{Conclusion}
\label{Conclusion}
This study proposes an algorithm grounded in OCP that effectively leverages Hutchinson-estimated diagonal second-order information. Theoretically, we establish its convergence rate under standard assumptions. Empirically, the algorithm demonstrates favorable performance in visual localization tasks, achieving accelerated convergence, enhanced generalization capability, and robustness to step-size hyperparameter tuning. A key contribution of this work lies in the successful integration of OCP method into deep learning, validating its effectiveness and feasibility while providing valuable insights for cross-disciplinary collaboration between these fields. Future work will focus on verifying its generalization capacity across more diverse datasets and architectures.

\printcredits

\section*{Data availability}
The data that support the findings of this study are available from the following sources:
\begin{itemize}
	\item[]
	$\bullet$ The KITTI vision benchmark dataset~\citep{ref29}, which was used for the visual localization experiments in this study, is publicly available and can be accessed at its official website: \url{https://www.cvlibs.net/datasets/kitti/}.
	
	$\bullet$ The source code, algorithm implementations, and simulation results generated during this study are available from the corresponding author upon reasonable request.
\end{itemize}

\section*{Declaration of competing interest}
The authors declare that they have no known competing financial interests or personal relationships that could have appeared to influence the work reported in this paper.

\begin{comment}
\section*{Acknowledgements}
This work was supported by the Original Exploratory Program Project of National Natural Science Foundation of China 62450004, the Joint Funds of the National Natural Science Foundation of China U23A20325,  and the Major Basic Research of Natural Science Foundation of Shandong Province ZR2021ZD14.
\end{comment}

\appendix
\section{The analysis of the convergence rate}
The analysis for the Diag-OCP adopts and extends the proof framework established for generalized Adam-type algorithms \citep{ref27}. This framework is particularly well-suited for analyzing adaptive optimization algorithms that utilize exponential moving averages of gradients and second-order moments. The main deviation and extension of our framework lies in how to handle the OCP-based update rule \eqref{eq16} and how to integrate the diagonal Hessian approximation $H_k$, which requires careful handling to establish boundedness and convergence under assumptions.
\subsection{Equivalence of the Algorithm}
\begin{lemma} \label{lemma1}
	The closed-form solution of the Algorithm \eqref{eq16}, when $\|I - M \hat{D}_{k}\| < 1$, is equivalent to:
	\begin{equation}
		x_{k+1} = x_{k} - (I - (I - M \hat{D}_{k})^{k+1}) \hat{D}_{k}^{-1} \hat{m}_k \label{aeq2}
	\end{equation}
\end{lemma}

\begin{proof}
	For Algorithm \eqref{eq16}, it is necessary to explicitly express $\phi_{k}(x_{k})$. Let:
	\begin{equation}
		\begin{aligned}
			&A = I - M \hat{D}_{k} \\
			&b = M \hat{m}_k ,
		\end{aligned}
	\end{equation}
	thus, the inner loop recursion is given by:
	\begin{equation}
		\phi_{i}(x_{k}) = b + A\phi_{i-1}(x_{k}).
	\end{equation}
	Expanding the first few terms yields:
	\begin{equation}
		\begin{aligned}
			\phi_{0}(x_{k}) &= b, \\
			\phi_{1}(x_{k}) &= b + Ab, \\
			\phi_{2}(x_{k}) &= b + A(b + Ab) = b + Ab + A^2b, \\
			&\;\;\vdots \\
			\phi_{l}(x_{k}) &= \sum_{i=0}^l A^i b,
		\end{aligned}
	\end{equation}
	therefore, the result of the inner loop is the partial sum of a matrix power series:
	\begin{equation}
		\phi_{l}(x_{k}) = \sum_{i=0}^l A^i b. \label{aeq3}
	\end{equation}
	When $\|A\| < 1$, Equation \eqref{aeq3} has the following identity:
	\begin{equation}
		\sum_{i=0}^l A^i = (I - A^{l+1})(I - A)^{-1}.
	\end{equation}
	Substituting the original problem variables, by letting $A = I - M \hat{D}_{k}$ and $b = M \hat{m}_k$, yields:
	\begin{equation}
		\begin{aligned}
			\phi_{l}(x_{k}) &= (I - A^{l+1})(I - A)^{-1}b \\
			&= (I - (I - M \hat{D}_{k})^{l+1}) \cdot (M \hat{D}_{k})^{-1} M \hat{m}_k \\
			&= (I - (I - M \hat{D}_{k})^{l+1}) \hat{D}_{k}^{-1} \hat{m}_k.
		\end{aligned}
	\end{equation}
	In the full iteration, $l = k$, and thus the original algorithm has the following update formula:
	\begin{equation}
		x_{k+1} = x_{k} - (I - (I - M \hat{D}_{k})^{k+1}) \hat{D}_{k}^{-1} \hat{m}_k .
	\end{equation}
	In summary, when $\|I - M \hat{D}_{k}\| < 1$, Algorithm \eqref{eq16} and \eqref{aeq2} are equivalent.
\end{proof}

\subsection{Construction of an Auxiliary Sequence}
To handle the momentum term $m_k$, define an auxiliary sequence $z_k$:
\begin{equation}
	z_k = x_k + \frac{\beta_1}{1 - \beta_1} (x_k - x_{k-1}) \quad k \geq 1 ,
\end{equation}
here, the initialization is set as $x_0 = x_1$, and thus $z_1 = x_1$.
This sequence encodes the momentum information into the current state, simplifying the analysis.

By definition:
\begin{equation}
	\begin{aligned}
		&z_{k+1} = x_{k+1} + \frac{\beta_1}{1 - \beta_1}(x_{k+1} - x_k) \\
		&z_k = x_k + \frac{\beta_1}{1 - \beta_1}(x_k - x_{k-1}) ,
	\end{aligned}
\end{equation}
substituting into the algorithm update $x_{k+1} = x_k - \Gamma_k \hat{m}_k$, where $\Gamma_k = (I - (I - M \hat{D}_k)^{k+1}) \hat{D}_k^{-1}$:
\begin{equation}
	\begin{aligned}
		z_{k+1} - z_k &= [ x_{k+1} + \frac{\beta_1}{1 - \beta_1}(x_{k+1} - x_k) ] \\
		& \quad - [ x_k + \frac{\beta_1}{1 - \beta_1}(x_k - x_{k-1}) ] \\
		&= (x_{k+1} - x_k) \\
		& \quad + \frac{\beta_1}{1 - \beta_1}[(x_{k+1} - x_k) - (x_k - x_{k-1})] \\
		&= -\Gamma_k \hat{m}_k + \frac{\beta_1}{1 - \beta_1} [ -\Gamma_k \hat{m}_k + \Gamma_{k-1} \hat{m}_{k-1}] \\
		&= -\frac{1}{1 - \beta_1} \Gamma_k \hat{m}_k + \frac{\beta_1}{1 - \beta_1} \Gamma_{k-1} \hat{m}_{k-1} \label{aeq4} .
	\end{aligned}
\end{equation}
At initialization, $m_0 = 0$, thus $\hat{m}_0 = 0$, and therefore $\Gamma_0 \hat{m}_0 = 0$. 

\subsection{Descent in Function Value}
Using \ref{ass:a1}, analyze the decrease of $f(z_k)$:
\begin{equation}
	f(z_{k+1}) \leq f(z_k) + \langle \nabla f(z_k), z_{k+1} - z_k \rangle + \frac{L}{2} \|z_{k+1} - z_k\|^2 .
\end{equation}
Take the expectation:
\begin{equation}
	\begin{aligned}
		\mathbb{E}[f(z_{k+1})] &\leq \mathbb{E}[f(z_k)] + \mathbb{E}[\langle \nabla f(z_k), z_{k+1} - z_k \rangle] \\
		& \quad + \frac{L}{2} \mathbb{E}[\|z_{k+1} - z_k\|^2] \label{aeq19} .
	\end{aligned}
\end{equation}
Substitute the result from Equation \eqref{aeq4} :
\begin{equation}
	\begin{aligned}
		&\mathbb{E}[\langle \nabla f(z_k), z_{k+1} - z_k \rangle] \\
		&= \mathbb{E} [\langle \nabla f(z_k), -\frac{1}{1-\beta_1} \Gamma_k \hat{m}_k + \frac{\beta_1}{1-\beta_1} \Gamma_{k-1} \hat{m}_{k-1} \rangle ] \\
		&= \mathbb{E}[ -\frac{1}{1 - \beta_{1}} \langle \nabla f(z_k) , \Gamma_k \hat{m}_k \rangle \\
		& \quad \quad + \frac{\beta_{1}}{1 - \beta_{1}} \langle \nabla f(z_k) , \Gamma_{k-1} \hat{m}_{k-1} \rangle]\label{aeq5} .
	\end{aligned}
\end{equation}
\subsection{Algebraic Decomposition of the Inner Product}
Expand $\hat{m}_k$ as follows:
\begin{equation}
	\begin{aligned}
		\hat{m}_k &= \frac{m_k}{1 - \beta_1^k} \\
		&= \frac{\beta_1 m_{k-1} + (1 - \beta_1) g_k}{1 - \beta_1^k} \\
		&= \frac{\beta_1 (1 - \beta_1^{k-1}) \hat{m}_{k-1} + (1 - \beta_1)(\nabla f(x_k) + \zeta_k)}{1 - \beta_1^k} \label{aeq6} .
	\end{aligned}
\end{equation}
For the term $\left\langle \nabla f(z_k) , \Gamma_k \hat{m}_k \right\rangle$ in Equation \eqref{aeq5}, substituting $\hat{m}_k$ yields:
\begin{equation}
	\begin{aligned}
		\langle \nabla f(z_k), \Gamma_k \hat{m}_k \rangle &= \frac{\beta_1 (1 - \beta_1^{k-1})}{1 - \beta_1^k} \langle \nabla f(z_k), \Gamma_k \hat{m}_{k-1} \rangle \\
		& \quad + \frac{1 - \beta_1}{1 - \beta_1^k} \langle \nabla f(z_k), \Gamma_k (\nabla f(x_k) + \zeta_k) \rangle \\
		&=\frac{\beta_1 (1 - \beta_1^{k-1})}{1 - \beta_1^k} \langle \nabla f(z_k), \Gamma_k \hat{m}_{k-1} \rangle \\
		& \quad + \frac{1 - \beta_1}{1 - \beta_1^k} (\langle \nabla f(z_k), \Gamma_k \nabla f(x_k)  \rangle \\
		& \quad + \langle \nabla f(z_k), \Gamma_k \zeta_k \rangle) \label{aeq7} .
	\end{aligned}
\end{equation}
The term $\langle \nabla f(z_k), \Gamma_k \nabla f(x_k) \rangle$ can further be decomposed as follows:
\begin{equation}
	\begin{aligned}
		\langle \nabla f(z_k), \Gamma_k \nabla f(x_k) \rangle &= \langle \nabla f(x_k), \Gamma_k \nabla f(x_k) \rangle \\
		& \quad + \langle \nabla f(z_k) - \nabla f(x_k), \Gamma_k \nabla f(x_k) \rangle .
	\end{aligned}
\end{equation}
Therefore, substituting Equation \eqref{aeq6} and Equation \eqref{aeq7} into Equation \eqref{aeq5} without expectation yields:
\begin{equation}
	\begin{aligned}
		&\langle \nabla f(z_k), z_{k+1} - z_k \rangle \\
		&=  -\frac{1}{1 - \beta_{1}} \left\langle \nabla f(z_k) , \Gamma_k \hat{m}_k \right\rangle \\
		& \quad + \frac{\beta_{1}}{1 - \beta_{1}} \left\langle \nabla f(z_k) , \Gamma_{k-1} \hat{m}_{k-1} \right\rangle \\
		&= -\frac{1}{1 - \beta_{1}} \big( \frac{\beta_1 (1 - \beta_1^{k-1})}{1 - \beta_1^k} \langle \nabla f(z_k), \Gamma_k \hat{m}_{k-1} \rangle \\
		& \quad + \frac{1 - \beta_1}{1 - \beta_1^k} ( \langle \nabla f(z_k), \Gamma_k \nabla f(x_k)  \rangle \\
		& \quad + \langle \nabla f(z_k), \Gamma_k \zeta_k \rangle ) \big) \\
		& \quad + \frac{\beta_{1}}{1 - \beta_{1}} \left\langle \nabla f(z_k) , \Gamma_{k-1} \hat{m}_{k-1} \right\rangle \\
		&= -\frac{\beta_{1} (1-\beta_{1}^{k-1})}{(1 - \beta_{1}) (1 - \beta_1^k)} \langle \nabla f(z_k), \Gamma_k \hat{m}_{k-1} \rangle \\
		& \quad - \frac{1}{1 - \beta_{1}^k} \left\langle \nabla f(x_k) , \Gamma_{k} \nabla f(x_k) \right\rangle \\
		& \quad - \frac{1}{1 - \beta_{1}^k} \left\langle \nabla f(z_k) - \nabla f(x_k), \Gamma_{k} \nabla f(x_k) \right\rangle \\
		& \quad - \frac{1}{1 - \beta_{1}^k} \left\langle \nabla f(z_k), \Gamma_{k} \zeta_k \right\rangle \\
		& \quad + \frac{\beta_{1}}{1 - \beta_{1}} \left\langle \nabla f(z_k) , \Gamma_{k-1} \hat{m}_{k-1} \right\rangle \label{aeq8} .
	\end{aligned}
\end{equation}

Extracting the main descent term $- \frac{1}{1 - \beta_{1}^k} \left\langle \nabla f(x_k), \Gamma_{k} \nabla f(x_k) \right\rangle$ from Equation \eqref{aeq8}, we define the remaining terms as the residual term $R_k$, yielding:
\begin{equation}
	\langle \nabla f(z_k), z_{k+1} - z_k \rangle = - \frac{1}{1 - \beta_{1}^k} \left\langle \nabla f(x_k) , \Gamma_{k} \nabla f(x_k) \right\rangle + R_{k} ,
\end{equation}
here, $R_k$ is given by:
\begin{equation}
	\begin{aligned}
		R_{k} &= -\frac{\beta_{1} (1-\beta_{1}^{k-1})}{(1 - \beta_{1}) (1 - \beta_1^k)} \langle \nabla f(z_k), \Gamma_k \hat{m}_{k-1} \rangle \\
		& \quad - \frac{1}{1 - \beta_{1}^k} \left\langle \nabla f(z_k) - \nabla f(x_k), \Gamma_{k} \nabla f(x_k) \right\rangle \\
		& \quad - \frac{1}{1 - \beta_{1}^k} \left\langle \nabla f(z_k), \Gamma_{k} \zeta_k \right\rangle \\
		& \quad + \frac{\beta_{1}}{1 - \beta_{1}} \left\langle \nabla f(z_k) , \Gamma_{k-1} \hat{m}_{k-1} \right\rangle .
	\end{aligned}
\end{equation}

Now, we begin to reorganize the residual terms.
First, combine these two terms to form $R_k^{(1)}$:
\begin{equation}
	\begin{aligned}
		R_k^{(1)} &= \frac{\beta_{1}}{1 - \beta_{1}} \left\langle \nabla f(z_k) , \Gamma_{k-1} \hat{m}_{k-1} \right\rangle \\
		& \quad - \frac{\beta_{1} (1-\beta_{1}^{k-1})}{(1 - \beta_{1}) (1 - \beta_1^k)} \langle \nabla f(z_k), \Gamma_k \hat{m}_{k-1} \rangle \\
		& = \frac{\beta_1}{1 - \beta_1} \big( \langle \nabla f(z_k), \Gamma_{k-1} \hat{m}_{k-1} \rangle \\
		& \quad - \frac{1 - \beta_1^{k-1}}{1 - \beta_1^k} \langle \nabla f(z_k), \Gamma_k \hat{m}_{k-1} \rangle \big) \\
		& = \frac{\beta_1}{1 - \beta_1} \langle \nabla f(z_k), ( \Gamma_{k-1} - \frac{1 - \beta_1^{k-1}}{1 - \beta_1^k} \Gamma_k ) \hat{m}_{k-1} \rangle .
	\end{aligned}
\end{equation}
\begin{comment}
	Note that in $R_k^{(1)}$, $\hat{m}{k-1}$ is the bias-corrected past momentum, but it does not explicitly include the bias with respect to the current gradient $\nabla f(x_k)$. However, we can further decompose the bias within $R_k^{(1)}$. Specifically, we write $\hat{m}{k-1}$ as:
	\begin{equation}
		\hat{m}_{k-1} = \nabla f(x_k) + (\hat{m}_{k-1} - \nabla f(x_k))
	\end{equation}
	Then:
	\begin{equation}
		\begin{aligned}
			R_k^{(1)} &= \frac{\beta_1}{1 - \beta_1} \langle \nabla f(z_k), ( \Gamma_{k-1} - \frac{1 - \beta_1^{k-1}}{1 - \beta_1^k} \Gamma_k ) \nabla f(x_k) \rangle \\
			&\quad + \frac{\beta_1}{1 - \beta_1} \langle \nabla f(z_k), ( \Gamma_{k-1} - \frac{1 - \beta_1^{k-1}}{1 - \beta_1^k} \Gamma_k ) \\
			& \quad \cdot (\hat{m}_{k-1} - \nabla f(x_{k})) \rangle
		\end{aligned}
	\end{equation}
\end{comment}
Then, define $R_k^{(2)}$ as:
\begin{equation}
	R_{k}^{(2)} = - \frac{1}{1 - \beta_{1}^k} \left\langle \nabla f(z_k), \Gamma_{k} \zeta_k \right\rangle .
\end{equation}
Finally, $R_k^{(3)}$ is:
\begin{equation}
	R_{k}^{(3)} = - \frac{1}{1 - \beta_{1}^k} \left\langle \nabla f(z_k) - \nabla f(x_k), \Gamma_{k} \nabla f(x_k) \right\rangle .
\end{equation}
\begin{comment}
	Here, $R_k^{(1a)}$ is the coupling between the weight matrix oscillation and the current gradient, $R_k^{(1b)}$ is the momentum bias term, $R_k^{(2)}$ is the gradient noise term, and $R_k^{(3)}$ is the Lipschitz coupling term.
\end{comment}

\subsection{Bounding the Residual Terms}
\subsubsection{Bound for $R_k^{(1)}$}
First, by applying the Cauchy–Schwarz inequality and \ref{ass:a2}:
\begin{equation}
	|R_k^{(1)}| \leq \frac{\beta_1}{1 - \beta_1} \|\nabla f(z_k)\| \cdot \|\Gamma_{k-1} - \frac{1 - \beta_1^{k-1}}{1 - \beta_1^k} \Gamma_k\| \cdot \|\hat{m}_{k-1}\|.
\end{equation}
By \ref{ass:a2}, we have $\|\nabla f(z_k)\| \leq \mathcal{H}$, and also $\|\hat{m}_{k-1}\| = \frac{\|m_{k-1}\|}{1 - \beta_1^{k-1}} \leq \frac{\mathcal{H}_g}{1 - \beta_1^{k-1}}$
(where $\|m_{k-1}\| \leq \mathcal{H}_g$ can be derived from \ref{ass:a2}’s condition $\| g(x_k) \| \leq \mathcal{H}_g$ by recursion starting from $m_0$), therefore:
\begin{equation}
	|R_k^{(1)}| \leq \frac{\beta_1}{1 - \beta_1} \cdot \mathcal{H} \cdot \|\Gamma_{k-1} - \frac{1 - \beta_1^{k-1}}{1 - \beta_1^k} \Gamma_k\| \cdot \frac{\mathcal{H}_g}{1 - \beta_1^{k-1}}
\end{equation}

\begin{lemma} \label{lemma2}
	By \ref{ass:a4}: $\mu I \preceq H_k \preceq G_d I$, it follows that $\|\hat{D}_k\| \leq G_d$, $\|\hat{D}_k^{-1}\| \leq \mu^{-1}$.
\end{lemma}

\begin{proof}
	By \ref{ass:a4}, it follows directly that:
	\begin{equation}
		\mu \leq [H_k]_{ii} \leq G_d.
	\end{equation}
	The update rule for $D_k$ is the exponential moving average:
	\begin{equation}
		D_k = \beta_2 D_{k-1} + (1 - \beta_2) H_k.
	\end{equation}
	With the initialization $D_0 = 0$, therefore:
	\begin{equation}
		D_k = (1 - \beta_2) \sum_{j=1}^k \beta_2^{k-j} H_j.
	\end{equation}
	Its diagonal elements are:
	\begin{equation}
		[D_k]_{ii} = (1 - \beta_2) \sum_{j=1}^k \beta_2^{k-j} [H_j]_{ii}.
	\end{equation}
	Since $[H_j]_{ii} \in [\mu, G_d]$, and $(1 - \beta_2) \cdot \sum_{j=1}^k \beta_2^{k-j} = 1 - \beta_2^k$, it follows that:
	\begin{equation}
		\mu (1 - \beta_2^k) \leq [D_k]_{ii} \leq G_d (1 - \beta_2^k).
	\end{equation}
	Dividing both sides by $1 - \beta_2^k$ (bias correction):
	\begin{equation}
		\mu \leq \frac{[D_k]_{ii}}{1 - \beta_2^k} = [\hat{D}_k]_{ii} \leq G_d.
	\end{equation}
	Final conclusion:
	\begin{equation}
		\mu I \preceq \hat{D}_k \preceq G_d I.
	\end{equation}
	Therefore, $\|\hat{D}_k\| \leq G_d$, $\|\hat{D}_k^{-1}\| \leq \mu^{-1}$.
\end{proof}

\begin{comment}
	\textbf{Corollary 1.} By Assumption A4: $\mu I \preceq H_k \preceq G_d I$, it follows that the diagonal Hessian matrix ${\rm diag} (\nabla^2 f(x_k))$ is positive definite and bounded, i.e., $\mu I \preceq {\rm diag} (\nabla^2 f(x_k)) \preceq G_d I$
	
	\textbf{The proof of the Corollary 1:} Since $H_k$ is an unbiased estimate of $\mathrm{diag}(\nabla^2 f(x_k))$, and $H_k$ satisfy:
	\begin{equation}
		\mu I \preceq H_k \preceq G_d I
	\end{equation}
	Therefore, ${\rm diag} (\nabla^2 f(x_k))$ also satisfy:
	\begin{equation}
		\mu I \preceq \mathbb{E} [H_k | x_k] = {\rm diag} (\nabla^2 f(x_k)) \preceq G_d I
	\end{equation}
\end{comment}

\begin{lemma} \label{lemma3}
	Under the conditions of Lemma \ref{lemma2}, if $\|I - M \hat{D}_k\| < 1$ holds, then it follows that $0 < M \leq \frac{\mu}{2G_d^2}$. Moreover, $\|I - M\hat{D}_k\|_2 \leq \rho < 1$ also holds, where $\rho = 1 - M\mu$.
\end{lemma}

\begin{proof}
	By Lemma \ref{lemma2}, $\mu I \preceq \hat{D}_k \preceq G_d I$, and the eigenvalues of $\hat{D}_k$ satisfy:
	\begin{equation}
		\mu \leq \lambda_i(\hat{D}_k) \leq G_d.
	\end{equation}
	Therefore:
	\begin{equation}
		1 - MG_d \leq 1 - M\lambda_i(\hat{D}_k) \leq 1 - M\mu.
	\end{equation}
	Thus:
	\begin{equation}
		\|I - M\hat{D}_k\|_2 \leq \max(|1 - MG_d|, |1 - M\mu|) \label{aeq9}.
	\end{equation}
	To ensure that $\|I - M \hat{D}_k\|_2 < 1$, it is required that:
	\begin{equation}
		-1 < 1 - MG_d \leq 1 - M\mu < 1.
	\end{equation}
	It follows that:
	\begin{equation}
		\begin{aligned}
			&0 < M < \frac{2}{G_d}  \\
			&0 < M < \frac{2}{\mu},
		\end{aligned}
	\end{equation}
	namely:
	\begin{equation}
		0 < M < \frac{2}{G_d}.
	\end{equation}
	Since $\mu \leq G_{d}$, the condition is further tightened to:
	\begin{equation}
		0 < M \leq \frac{\mu}{G_{d}^2} \leq \frac{1}{G_{d}} < \frac{2}{G_{d}}.
	\end{equation}
	This condition can be further tightened to:
	\begin{equation}
		0 < M \leq \frac{\mu}{2G_{d}^2} \label{aeq10}.
	\end{equation}
	Based on Equation \eqref{aeq10}, it can be derived that $1 - MG_d > 0$ and $1 - M\mu > 0$, which allows the removal of the absolute value signs.
	
	According to the assumption $M \leq \frac{\mu}{2G_d^2}$ and $\mu \leq G_d$, it follows that:
	\begin{equation}
		M \leq \frac{\mu}{2G_{d}^2} \leq \frac{1}{2G_{d}},
	\end{equation}
	therefore:
	\begin{equation}
		MG_d \leq \frac{1}{2} \implies 1 - MG_d \geq \frac{1}{2} > 0 \label{aeq11}.
	\end{equation}
	Similarly, it can be shown that:
	\begin{equation}
		M\mu \leq \frac{\mu^2}{2G_d^2} \leq \frac{1}{2}.
	\end{equation}
	The final conclusion is:
	\begin{equation}
		1 - M\mu \geq \frac{1}{2} > 0 \label{aeq12}.
	\end{equation}
	From Equations \eqref{aeq11}, \eqref{aeq12}, and \eqref{aeq9}, it follows that:
	\begin{equation}
		\|I - M\hat{D}_k\|_2 \leq \rho < 1, \quad \rho = 1 - M\mu.
	\end{equation}
\end{proof}

Now, consider the expression for $\Gamma_k$: $\Gamma_k = (I - (I - M\hat{D}_k)^{k+1})\hat{D}_k^{-1}$. By Lemma \ref{lemma2} and Lemma \ref{lemma3}, we have $\|\hat{D}_k\| \leq G_d$, $\|\hat{D}_k^{-1}\| \leq \mu^{-1}$, and $\|I - M\hat{D}_k\| \leq \rho < 1$. Therefore:
\begin{equation}
	\|(I - M\hat{D}_k)^{k+1}\| \leq \rho^{k+1} \label{aeq15}.
\end{equation}
Hence:
\begin{equation}
	\begin{aligned}
		\|\Gamma_k\| &\leq \|I - (I - M\hat{D}_k)^{k+1}\| \cdot \|\hat{D}_k^{-1}\| \\
		&\leq (\|I\| + \| (I - M\hat{D}_{k})^{k+1} \|) \cdot \|\hat{D}_k^{-1}\| \\
		&\leq (1 + \rho^{k+1}) \mu^{-1} \label{aeq18}.
	\end{aligned}
\end{equation}
Similarly, $\| \Gamma_{k-1} \| \leq (1 + \rho^k) \mu^{-1}$. Moreover, since $\beta_1 < 1$, when $k$ is sufficiently large, $\frac{1 - \beta_1^{k-1}}{1 - \beta_1^k} \approx 1$. Therefore, the continuity between $\Gamma_k$ and $\Gamma_{k-1}$ can also be exploited, therefore:
\begin{equation}
	\begin{aligned}
		\| \Gamma_{k-1} - \frac{1 - \beta_1^{k-1}}{1 - \beta_1^k} \Gamma_k \| &\leq \| \Gamma_{k-1} \| + \frac{1 - \beta_1^{k-1}}{1 - \beta_1^k} \| \Gamma_k \| \\
		&\leq (1 + \rho^k) \mu^{-1} + (1 + \rho^{k+1}) \mu^{-1} \\
		&\leq 2(1 + \rho^k) \mu^{-1}.
	\end{aligned}
\end{equation}
Thus:
\begin{equation}
	|R_k^{(1)}| \leq \frac{\beta_1}{1 - \beta_1} \cdot \mathcal{H} \cdot 2(1 + \rho^k) \mu^{-1} \cdot \frac{\mathcal{H}_g}{1 - \beta_1^{k-1}}.
\end{equation}
Since $\rho < 1$, $\rho^k$ decays exponentially, and $1 - \beta_1^{k-1} \geq 1 - \beta_1$ (for $k \geq 2$), it follows that:
\begin{equation}
	\mathbb{E}[|R_k^{(1)}|] \leq \frac{\beta_1}{1 - \beta_1} \cdot \mathcal{H} \cdot 2(1 + \rho^k) \mu^{-1} \cdot \frac{\mathcal{H}_g}{1 - \beta_1} = \mathcal{O}(\beta_1 \cdot (1 + \rho^k)).
\end{equation} 
The factor $1 + \rho^k$ can be interpreted as follows: when $k$ is small, $\rho^k$ is not negligible, so the residual term $R_k^{(1)}$ is significantly influenced by the momentum history and the dynamic changes of the weight matrix; as $k$ becomes large, $\rho^k \to 0$, and the residual term converges to a steady-state value of $\mathcal{O}(\beta_1)$.

\subsubsection{Bound for $R_k^{(2)}$}
Reanalyzing the noise coupling term:
\begin{equation}
	R_k^{(2)} = - \frac{1}{1 - \beta_{1}^k} \langle \nabla f(z_k), \Gamma_k \zeta_k\rangle.
\end{equation}
By the Cauchy–Schwarz inequality, it holds that:
\begin{equation}
	|R_k^{(2)}| \leq \frac{1}{1 - \beta_{1}^k} \cdot \|\nabla f(z_k)\| \cdot \|\Gamma_k\| \cdot \|\zeta_k\|,
\end{equation}
hence:
\begin{equation}
	\mathbb{E}[|R_k^{(2)}|] \leq \frac{1}{1 - \beta_{1}^k} \mathbb{E}[\|\nabla f(z_k)\| \cdot \|\Gamma_k\| \cdot \|\zeta_k\|] \label{aeq13}.
\end{equation}
From Equation \eqref{aeq18}, it can be deduced that:
\begin{equation}
	\|\Gamma_k\| \leq (1 + \rho^{k+1}) \mu^{-1}.
\end{equation}
Controlling the conditional expectation of $\| \Gamma_k \|^2$. Note that the above bound is deterministic and does not depend on the noise (since in Lemma \ref{lemma3}, $\rho$ is a constant and the step size $M$ is fixed). Therefore, it holds that:
\begin{equation}
	\| \Gamma_k \|^2 \leq \left( (1 + \rho^{k+1}) \mu^{-1} \right)^2.
\end{equation}
Since the right-hand side is deterministic, the conditional expectation satisfies:
\begin{equation}
	\mathbb{E}[ \| \Gamma_k \|^2 | \mathcal{F}_{k-1} ] \leq  (1 + \rho^{k+1})^2 \mu^{-2}.
\end{equation}
Since $\rho < 1$, it follows that $(1 + \rho^{k+1})^2 \leq 4$ (because when $k=0$, $\rho^1 = \rho \leq 1$, thus the maximum is $(1 + 1)^2 = 4$):
\begin{equation}
	\mathbb{E}[ \| \Gamma_k \|^2 | \mathcal{F}_{k-1} ] \leq \frac{4}{\mu^2}.
\end{equation}
For Equation \eqref{aeq13}, by \ref{ass:a2}, the gradient norm is bounded: $\|\nabla f(z_k)\| \leq \mathcal{H}$. Therefore:
\begin{equation}
	\mathbb{E}\left[\|\nabla f(z_k)\|\cdot \|\Gamma_k\|\cdot \|\zeta_k\|\right] \leq \mathcal{H}\cdot \mathbb{E}\left[\|\Gamma_k\|\cdot \|\zeta_k\|\right] \label{aeq14}.
\end{equation}
Applying the Cauchy–Schwarz inequality again in expectation (note: the expectation of the product of two random variables satisfies):
\begin{equation}
	\mathbb{E}[|XY|] \leq \sqrt{\mathbb{E}[X^2] \cdot \mathbb{E}[Y^2]}.
\end{equation}
Then, in Equation \eqref{aeq14}:
\begin{equation}
	\mathbb{E} \left[ \|\Gamma_k\| \cdot \|\zeta_k\| \right] \leq \sqrt{\mathbb{E} \left[ \|\Gamma_k\|^2 \right] \cdot \mathbb{E} \left[ \|\zeta_k\|^2 \right]}.
\end{equation}
By \ref{ass:a2}, the variance of the gradient noise is bounded: $\mathbb{E} [\|\zeta_k\|^2] \leq \sigma_g^2$. Therefore:
\begin{equation}
	\mathbb{E} \left[ \|\zeta_k\|^2 \right] \leq \sigma_g^2.
\end{equation}
At the same time, by the tower property of conditional expectation:
\begin{equation}
	\mathbb{E} \left[ \| \Gamma_k \|^2 \right] = \mathbb{E} \left[ \mathbb{E} \left[ \| \Gamma_k \|^2 |\mathcal{F}_{k-1} \right] \right] \leq \frac{4}{\mu^2}.
\end{equation}
Thus:
\begin{equation}
	\begin{aligned}
		\mathbb{E}\left[\|\nabla f(z_k)\|\cdot \|\Gamma_k\|\cdot \|\zeta_k\|\right] &\leq \mathcal{H}\cdot \mathbb{E}\left[\|\Gamma_k\|\cdot \|\zeta_k\|\right] \\
		&\leq \mathcal{H}\cdot \sqrt{\mathbb{E} \left[ \|\Gamma_k\|^2 \right] \cdot \mathbb{E} \left[ \|\zeta_k\|^2 \right]} \\
		&\leq \mathcal{H} \cdot \sqrt{\frac{4}{\mu^2} \cdot \sigma_g^2} \\
		&\leq \frac{2 \sigma_g \mathcal{H}}{\mu}.
	\end{aligned}
\end{equation}
In summary:
\begin{equation}
	\mathbb{E}[|R_k^{(2)}|] \leq \frac{1}{1 - \beta_{1}^k} \cdot \frac{2 \sigma_g \mathcal{H}}{\mu}.
\end{equation}
Since $\beta_1 < 1$, when $k \geq 1$, $1 - \beta_1^k \geq 1 - \beta_1$. Therefore:
\begin{equation}
	\mathbb{E}[|R_k^{(2)}|] \leq \frac{1}{1 - \beta_1} \cdot \frac{2 \sigma_g \mathcal{H}}{\mu} = \mathcal{O}(1).
\end{equation}

\subsubsection{Bound for $R_k^{(3)}$}
\begin{equation}
	|R_k^{(3)}| \leq \frac{1}{1 - \beta_1^k} \|\nabla f(z_k) - \nabla f(x_k)\| \cdot \|\Gamma_k\| \cdot \|\nabla f(x_k)\|.
\end{equation}
By the Lipschitz continuity of the gradient:
\begin{equation}
	\|\nabla f(z_k) - \nabla f(x_k)\| \leq L \|z_k - x_k\|.
\end{equation}
By the definition of the auxiliary sequence:
\begin{equation}
	z_k - x_k = \frac{\beta_1}{1 - \beta_1} (x_k - x_{k-1}) = -\frac{\beta_1}{1 - \beta_1} \Gamma_{k-1} \hat{m}_{k-1}.
\end{equation}
Thus:
\begin{equation}
	\begin{aligned}
		\|z_k - x_k\| &\leq \frac{\beta_1}{1 - \beta_1} \|\Gamma_{k-1}\| \cdot \|\hat{m}_{k-1}\| \\
		&\leq \frac{\beta_1}{1 - \beta_1} (1 + \rho^k) \mu^{-1} \cdot \frac{\mathcal{H}_g}{1 - \beta_1^{k-1}}.
	\end{aligned}
\end{equation}
Moreover, since $\|\nabla f(x_k)\| \leq \mathcal{H}$, $\|\Gamma_k\| \leq (1 + \rho^{k+1}) \mu^{-1}$, it follows that:
\begin{equation}
	\begin{aligned}
		|R_k^{(3)}| &\leq \frac{1}{1 - \beta_1^k} \cdot L \cdot ( \frac{\beta_1}{1 - \beta_1} (1 + \rho^k) \mu^{-1} \cdot \frac{\mathcal{H}_g}{1 - \beta_1^{k-1}} ) \\
		& \quad \cdot (1 + \rho^{k+1}) \mu^{-1} \cdot \mathcal{H}.
	\end{aligned}
\end{equation}
By simplification, and noting that when $k \geq 2$, $1 - \beta_1^{k-1} \geq 1 - \beta_1$ and $1 - \beta_1^k \geq 1 - \beta_1$, it follows that:
\begin{equation}
	\begin{aligned}
		\mathbb{E}[|R_k^{(3)}|] &\leq \frac{1}{(1 - \beta_1)^2} \cdot L \cdot \frac{\beta_1}{1 - \beta_1} \\
		& \quad \cdot (1 + \rho^k) \cdot (1 + \rho^{k+1})\mu^{-2} \cdot \mathcal{H}_g \cdot \mathcal{H} \\
		&= \mathcal{O} (\beta_1 \cdot \rho^{2k}).
	\end{aligned}
\end{equation}
The above simplified bound indicates that the order of the residual term is proportional to $\beta_{1}$, with its decay rate controlled by $\rho^{2k}$. Similarly, since $\rho < 1$, this term remains controllable after summation.

\subsection{Bounding the Lipschitz Smoothing Term}
In the analysis of function value descent, there is also a second-order term arising from the Lipschitz smoothness:
\begin{equation}
	\frac{L}{2} \| z_{k+1} - z_k \|^2.
\end{equation}
From the update rule, we have
\begin{equation}
	z_{k+1} - z_k = -\frac{1}{1-\beta_1} \Gamma_k \hat{m}_k + \frac{\beta_1}{1-\beta_1} \Gamma_{k-1} \hat{m}_{k-1},
\end{equation}
which implies
\begin{equation}
	\| z_{k+1} - z_k \| \leq \frac{1}{1-\beta_1} \| \Gamma_k \| \|\hat{m}_k \| + \frac{\beta_1}{1-\beta_1} \| \Gamma_{k-1} \| \|\hat{m}_{k-1} \|.
\end{equation}
Using $\| \Gamma_k \| \leq (1+\rho^{k+1}) \mu^{-1}$ and $\| \hat{m}_k \| \leq \mathcal{H}_g / (1-\beta_1^k)$, we obtain
\begin{equation}
	\| z_{k+1} - z_k \| \leq \frac{(1+\rho^{k+1})\mathcal{H}_g}{(1-\beta_1)(1-\beta_1^k)\mu} + \frac{\beta_1(1+\rho^k)\mathcal{H}_g}{(1-\beta_1)(1-\beta_1^{k-1})\mu}.
\end{equation}
Squaring both sides yields
\begin{equation}
	\begin{aligned}
		\|z_{k+1} - z_k\|^2 &\leq 2(\frac{(1 + \rho^{k+1})\mathcal{H}_g}{(1 - \beta_1)(1 - \beta_1^k)\mu})^2 \\
		& \quad + 2(\frac{\beta_1(1 + \rho^k)\mathcal{H}_g}{(1 - \beta_1)(1 - \beta_1^{k-1})\mu})^2.
	\end{aligned}
\end{equation}
Taking expectations, and noting that when $k \geq 2$, the denominator has a lower bound (which depends on $\beta_1$), it follows that:
\begin{equation}
	\mathbb{E}[ \|z_{k+1} - z_k\|^2 ] \leq \mathcal{O}(\rho^{2k}).
\end{equation}
Therefore, the expectation of this term is $\mathcal{O}(\rho^{2k})$, which is summable and thus convergent.

\subsection{Final Summation and Convergence Rate Analysis}
Summing the expected inequality for function value descent from $k=1$ to $T$:
\begin{equation}
	\begin{aligned}
		\mathbb{E}[f(z_{T+1}) - f(z_1)] &\leq \sum_{k=1}^{T} \mathbb{E} [\langle \nabla f(z_k), z_{k+1} - z_k \rangle] \\
		& \quad + \frac{L}{2} \sum_{k=1}^{T} \mathbb{E} [\|z_{k+1} - z_k\|^2].
	\end{aligned}
\end{equation}
According to the previous decomposition, it holds that
\begin{equation}
	\begin{aligned}
		&\mathbb{E} [\langle \nabla f(z_k), z_{k+1} - z_k \rangle] \\
		& \quad = \mathbb{E} [ -\frac{1}{1 - \beta_1^k} \langle \nabla f(x_k), \Gamma_k \nabla f(x_k) \rangle ] + \mathbb{E} [R_k].
	\end{aligned}
\end{equation}

\begin{lemma} \label{lemma4}
	For the term $\langle \nabla f(x_k), \Gamma_k \nabla f(x_k) \rangle$, it holds that:
	\begin{equation}
		\langle \nabla f(x_k), \Gamma_k \nabla f(x_k) \rangle \geq C_1 \|\nabla f(x_k)\|^2,
	\end{equation}
	where $C_1 > 0$.
\end{lemma} 

\begin{proof}
	More specifically, for $\Gamma_k = (I - (I - M\hat{D}_k)^{k+1}) \hat{D}_k^{-1}$, based on Equation \eqref{aeq15} where $\|(I - M\hat{D}_k)^{k+1}\| \leq \rho^{k+1}$ and using norm inequalities, it follows that:
	\begin{comment}
		For any vector $\mathcal{V}$,
		\begin{equation}
			\|(I - M\hat{D}_k)^{k+1} \mathcal{V}\| \leq \rho^{k+1}\| \mathcal{V} \|
		\end{equation}
		Then:
		\begin{equation}
			-\rho^{k+1}\| \mathcal{V} \|^2 \leq \mathcal{V}^\top(I - M\hat{D}_k)^{k+1} \mathcal{V} \leq \rho^{k+1}\| \mathcal{V} \|^2
		\end{equation}
		Thus:
		\begin{equation}
			\mathcal{V}^\top[I - (I - M\hat{D}_k)^{k+1}] \mathcal{V} \geq (1 - \rho^{k+1})\| \mathcal{V} \|^2
		\end{equation}	
	\end{comment}
	
	Letting $\mathcal{V} = \hat{D}_k^{-1/2} \nabla f(x_k)$:
	\begin{equation}
		\begin{aligned}
			&\langle \nabla f(x_k), \Gamma_k \nabla f(x_k) \rangle \\
			&\quad = \nabla f(x_k)^T [I - (I - M\hat{D}_k)^{k+1}] \hat{D}_k^{-1} \nabla f(x_k) \\
			& \quad = \mathcal{V}^T \hat{D}_k^{1/2} [I - (I - M\hat{D}_k)^{k+1}] \hat{D}_k^{-1/2} \mathcal{V} \\
			& \quad = \mathcal{V}^T [I - \hat{D}_k^{1/2}(I - M\hat{D}_k)^{k+1} \hat{D}_k^{-1/2}] \mathcal{V}.
		\end{aligned}
	\end{equation}
	Since $\|(I - M \hat{D}_k)^{k+1}\| \leq \rho^{k+1}$ and similarity transformations preserve the spectral norm, we have
	\begin{equation}
		\|\hat{D}_k^{1/2}(I - M\hat{D}_k)^{k+1} \hat{D}_k^{-1/2}\| \leq \rho^{k+1}.
	\end{equation}
	It follows that
	\begin{equation}
		\begin{aligned}
			&|\mathcal{V}^T \hat{D}_k^{1/2}(I - M\hat{D}_k)^{k+1} \hat{D}_k^{-1/2} \mathcal{V}| \\
			& \quad \leq \|\hat{D}_k^{1/2}(I - M\hat{D}_k)^{k+1} \hat{D}_k^{-1/2}\| \cdot \| \mathcal{V} \|^2 \\
			& \quad \leq \rho^{k+1} \| \mathcal{V} \|^2 \label{aeq16}.
		\end{aligned}
	\end{equation}
	Hence,
	\begin{equation}
		\begin{aligned}
			&\mathcal{V}^T [I - \hat{D}_k^{1/2}(I - M\hat{D}_k)^{k+1} \hat{D}_k^{-1/2}] \mathcal{V} \\
			& \quad = \| \mathcal{V} \|^2 - \mathcal{V}^T \hat{D}_k^{1/2}(I - M\hat{D}_k)^{k+1} \hat{D}_k^{-1/2} \mathcal{V} \\
			& \quad \quad \geq \| \mathcal{V} \|^2 - \rho^{k+1} \| \mathcal{V} \|^2 \\
			& \quad \quad \quad = (1 - \rho^{k+1}) \| \mathcal{V} \|^2 \label{aeq17}.
		\end{aligned}
	\end{equation}
	By setting $\mathcal{V} = \hat{D}_k^{-1/2} \nabla f(x_k)$, it follows that:
	\begin{equation}
		\| \mathcal{V} \|^2 = \mathcal{V}^T \mathcal{V} = \nabla f(x_k)^T \hat{D}_k^{-1} \nabla f(x_k).
	\end{equation}
	Since $\hat{D}_k \preceq G_d I$, it follows that $\hat{D}_k^{-1} \succeq G_d^{-1} I$. Therefore:
	\begin{equation}
		\nabla f(x_k)^T \hat{D}_k^{-1} \nabla f(x_k) \geq G_d^{-1} \| \nabla f(x_k) \|^2.
	\end{equation}
	Substituting into Equation \eqref{aeq17} yields:
	\begin{equation}
		\langle \nabla f(x_k), \Gamma_k \nabla f(x_k) \rangle \geq \frac{1 - \rho^{k+1}}{G_d} \| \nabla f(x_k) \|^2.
	\end{equation}
	At this point, it follows that $C_1 = \frac{1 - \rho^{k+1}}{G_d} > 0$.
\end{proof}

Therefore, the main descent term satisfies:
\begin{equation}
	-\frac{1}{1 - \beta_1^k} \langle \nabla f(x_k), \Gamma_k \nabla f(x_k) \rangle \leq -\frac{1-\rho^{k+1}}{(1 - \beta_1^k) G_d} \| \nabla f(x_k) \|^2.
\end{equation}
Since $0 \leq \beta_{1} < 1$ and $\rho < 1$, it holds that $1 - \beta_{1}^{k} \leq 1$ (taking the upper bound). When $k \geq 2$, $1 - \rho^{k+1} \geq 1 - \rho$ (taking the lower bound).
Using Equation \eqref{aeq16} and applying the ratio of the lower bound to the upper bound, the upper bound can be further relaxed as:
\begin{equation}
	-\frac{1}{1 - \beta_1^k} \langle \nabla f(x_k), \Gamma_k \nabla f(x_k) \rangle \leq -\frac{1-\rho}{G_d} \| \nabla f(x_k) \|^2.
\end{equation}
In expectation, the focus is on the summation:
\begin{equation}
	\begin{aligned}
		&\sum_{k=1}^T \mathbb{E} [ -\frac{1}{1 - \beta_1^k} \langle \nabla f(x_k), \Gamma_k \nabla f(x_k) \rangle ] \\ 
		& \quad \leq -C_2 \sum_{k=1}^T \mathbb{E} [ \| \nabla f(x_k) \|^2 ],
	\end{aligned}
\end{equation}
where, the constant term $C_2$ is independent of $k$, given by $C_2 = \frac{1 - \rho}{G_d} > 0$.

For the residual terms, based on the previous bounds, it holds that:
\begin{equation}
	\begin{aligned}
		&\sum_{k=1}^{T} \mathbb{E}[|R_k|] \\
		& \quad \leq \sum_{k=1}^{T} [ \mathcal{O}(\beta_1(1 + \rho^k)) + \mathcal{O}(1) + \mathcal{O}(\beta_1 \cdot \rho^{2k}) ] = \mathcal{O}(1).
	\end{aligned}
\end{equation}
Similarly, for the Lipschitz term, it follows that
\begin{equation}
	\frac{L}{2} \sum_{k=1}^{T} \mathbb{E} [ \|z_{k+1} - z_k\|^2 ] = \frac{L}{2} \sum_{k=1}^{T} \mathcal{O}(\rho^{2k}) = \mathcal{O}(1).
\end{equation}
Therefore, the overall sum in Equation \eqref{aeq19} satisfies
\begin{equation}
	\mathbb{E}[f(z_{T+1})] - f(z_1) \leq -C_2 \sum_{k=1}^{T} \mathbb{E} [ \|\nabla f(x_k)\|^2 ] + C_3,
\end{equation}
where, $C_3$ denotes a constant. Rearranging yields:
\begin{equation}
	\begin{aligned}
		C_2 \sum_{k=1}^{T} \mathbb{E} [ \|\nabla f(x_k)\|^2 ] &\leq f(z_1) - \mathbb{E}[f(z_{T+1})] + C_3 \\
		&\leq f(z_1) - f^* + C_3.
	\end{aligned}
\end{equation}
It follows that
\begin{equation}
	\frac{1}{T} \sum_{k=1}^{T} \mathbb{E} [ \| \nabla f(x_k) \|^2 ] \leq \frac{f(z_1) - f^* + C_3}{C_2T} = \mathcal{O} ( \frac{1}{T} ).
\end{equation}
Since the minimum gradient norm is necessarily less than or equal to the average value:
\begin{equation}
	\mathop{\mathrm{min}}\limits_{k \in [1,T]}\mathbb{E} [ \| \nabla f(x_k) \|^2 ] \leq \mathcal{O} ( \frac{1}{T} ).
\end{equation}

\bibliographystyle{cas-model2-names}

\bibliography{cas-refs}

\end{document}